\documentclass[aos, preprint]{imsart}

\usepackage{amsfonts,amsmath,amssymb,amsthm,bbm}
\usepackage[font=footnotesize,labelfont=bf]{caption}
\usepackage{subcaption}
\usepackage{float,mathtools}
\usepackage[shortlabels]{enumitem}
\usepackage{graphicx}
\usepackage{algorithmicx}
\usepackage[ruled]{algorithm}
\usepackage{algpseudocode}
\usepackage[numbers]{natbib}
\usepackage[colorlinks]{hyperref}
\usepackage{chngcntr}
\usepackage{apptools}
\AtAppendix{\counterwithin{theorem}{subsection}}        

\DeclarePairedDelimiter\ceil{\lceil}{\rceil}

\theoremstyle{remark}
\newenvironment{proof*}{\paragraph{Proof}}{\hfill$\blacksquare$}
\newenvironment{prooflemma}{\paragraph{\normalfont{\textit{Proof of Lemma}}}}{\hfill$\square$}
\newtheorem{remark}{Remark}
\newtheorem{theorem}{Theorem}[section]
\newtheorem{lemma}[theorem]{Lemma}
\newtheorem{proposition}[theorem]{Proposition}
\newtheorem{corollary}[theorem]{Corollary}
\newtheorem{definition}[theorem]{Definition}

\newcommand{\norm}[1]{\left\lVert#1\right\rVert}
\def\R{\mathbb{R}}
\def\N{\mathbb{N}}

\def\E{\mathbb{E}}

\def\Cov{\mathrm{Cov}}

\def\cF{\mathcal{F}}

\def\cI{\mathbbm{1}}

\newcommand{\Expect}[1]{\mathbb{E}\left[ #1 \right]}
\newcommand{\BootExpect}[1]{\mathbb{E}^*\left[ #1 \right]}

\newcommand{\DataExpect}[1]{\mathbb{E}_{M}\left[ #1 \right]}

\arxiv{1711.02834}

\begin{document}

\begin{frontmatter}
\title{Bootstrapping Generalization Error Bounds for Time Series}
\runtitle{}

\begin{aug}
\author{\fnms{Robert} \snm{Lunde}\ead[label=e1]{rlunde@andrew.cmu.edu}}
\and
\author{\fnms{Cosma Rohilla} \snm{Shalizi}\ead[label=e2]{cshalizi@cmu.edu}}\thanksref{t1}

\thankstext{t1}{Supported by grants from the Institute for New Economic Thinking (IN01100005, INO1400020) and from the National Science Foundation (DMS1207759, DMS1418124).}


\affiliation{Carnegie Mellon University}

\address{Department of Statistics\\
Carnegie Mellon University\\
Pittsburgh, PA 15213 \\
USA \\
\printead{e1}\\
\phantom{E-mail:\ }\printead*{e2}}

\end{aug}

\begin{abstract}
We consider the problem of finding confidence intervals for the risk of
forecasting the future of a stationary, ergodic stochastic process, using a
model estimated from the past of the process.  We show that a bootstrap
procedure provides valid confidence intervals for the risk, when the data
source is sufficiently mixing, and the loss function and the estimator are
suitably smooth.  Autoregressive (AR($d$)) models estimated by least squares
obey the necessary regularity conditions, even when mis-specified, and
simulations show that the finite-sample coverage of our bounds quickly
converges to the theoretical, asymptotic level.  As an intermediate step, we
derive sufficient conditions for asymptotic independence between empirical
distribution functions formed by splitting a realization of a stochastic
process, of independent interest.
\end{abstract}



\end{frontmatter}

\section{Introduction}
\label{Introduction}
Suppose we have observed data points $Y_1, \ldots, Y_t \equiv Y_{1:t}$ from a
stationary stochastic process $Y$, and want to predict $Y_{t+1}$ using a model
$\hat{f}_t$ estimated from the observations.  As in any other statistical
prediction problem, we want to be able to evaluate how well we will forecast,
and in particular make an estimate of the expected performance, or risk, of the
model, i.e., estimate the expected loss $L(Y_{t+1}, \hat{f}_t)$ for some
suitable loss function.  While several different notions of risk have been
proposed for time series \citep{Pestov-predictive-PAC,
  Kuznetsov-Mohri-non-stationary-mixing}, we follow
\citet{CRS-Leo-predictive-mixtures}, defining risk as the long-term average of
instantaneous expected losses, conditioned on the observations:
\begin{equation}
\label{ergodic risk}
R(\hat{f}_t) \equiv \lim_{m \rightarrow \infty} \frac{1}{m} \sum_{i=1}^m{\Expect{L(y_{t+i}, \hat{f}_t) | Y_{1:t}}}
\end{equation}
where it is understood that $\hat{f}_t$ is {\em estimated} using only
$Y_{1:t}$, but that when predicting $Y_{t+i}$, its {\em inputs} may come from
any time up to $t+i-1$.\footnote{Thus, for example, if we used an AR(1) model,
  with parameter $\theta$, the estimator $\hat{\theta}_t$ would be a function
  of $Y_{1:t}$ alone, but the prediction at time $t+i$ would be
  $\hat{\theta}_t(Y_{1:t})Y_{t+i-1}$.}  \citet{CRS-Leo-predictive-mixtures}
show that this is well-defined for ergodic sources.  Accepting this notion of
risk, how might we estimate it?  Suppose $\hat{f}_t$ takes as input some fixed
vector $Z_i = (Y_i, \ldots, Y_{i+d+1})$ and let $t_0 = t-d+1$.
Clearly, the in-sample performance of
$\hat{f}_t$ on the training data $Y_{1:t}$:
\begin{equation}
\hat{R}(\hat{f_t}) \equiv \frac{1}{t_0} \sum_{i=1}^{t_0}{L(z_i, \hat{f}_t)}
\end{equation}
will generally be overly optimistic an estimate of out-of-sample performance,
precisely because the model has been adjusted to fit that specific time series
(and not another).  We thus wish to know the ``generalization error'' of the
model $ \eta(\hat{f}_t) \equiv R(\hat{f_t}) - \hat{R}(\hat{f_t})$.  It is
particularly helpful to control the probability that $\eta(\hat{f}_t)$ exceeds
any given value, i.e., to probabilistically bound how much worse than its
in-sample performance $\hat{f}_t$ might really be.

In this paper, we give asymptotically-valid generalization error bounds for
time-series forecasting.  To do this, it is clearly sufficient to have a
consistent estimator for the $1 - \alpha$ quantile of the generalization error
$\eta_{1-\alpha}^*(\hat{f}_t)$, since that lets us give a probabilistic
guarantee for worst-case out-of-sample performance of our model:
\begin{align}
P(\hat{R}_t(\hat{f}_t) + \eta_{1-\alpha}^*(\hat{f}_t)) > 1 - \alpha
\end{align}
Following an idea proposed in \citet[sec.\ 8.3]{McDonald-thesis}, 
we achieve this by using the block bootstrap
\citep{Kunsch-bootstrap-for-stationary-obs} to directly simulate estimating a
model from a time series, and then evaluating the fitted model's ability to
forecast the continuation of the {\em same} realization of the process.  (See
Section \ref{sec:method}, Algorithm \ref{the-algorithm} for a precise statement
of the algorithm.)  Therefore, the main theoretical question that we consider
is bootstrap consistency of our estimator of the generalization error
$\eta^*(\hat{f}_t)$.  By this, we mean showing that the centered, normalized
bootstrap process,
\begin{equation}
\sqrt{t_0}\left[\eta^*(\hat{f}_t^*) -\BootExpect{\eta^*(\hat{f}_t)} \right] ~,
\end{equation}
converges in distribution, conditional on the data, to the same distribution as
the true sampling distribution,
\begin{equation}
\sqrt{t_0}\left[\widehat{\eta}(\hat{f}_t) -\Expect{\widehat{\eta}(\hat{f}_t)} \right] ~,
\end{equation}
where $\BootExpect{\cdot}$ denotes expectation with respect to the bootstrap
measure, and the precise definition of $\widehat{\eta}(\hat{f}_t)$ is postponed
to Section \ref{Main Theorem Statement}.

Our results focus on models $\hat{f}_t$ whose estimators are expressible as
plug-in estimators of some $d$-dimensional marginal distribution of the
stochastic process $Y$.  In this framework, we can also view the generalization
error itself as a functional of distribution functions, enabling the use of
weak convergence results from empirical process theory.  In contrast to the
constraints on the model and its estimator, essentially the only limit we put
on the data source $Y$ is that it has to be mixing (i.e., all correlations must
decay) at a polynomial rate.  In particular, we do {\em not} assume that the
model is in any way well-specified.  An important part of our proof is
establishing sufficient conditions for the asymptotic independence of empirical
processes formed by splitting one realization of a stochastic process.  This
was implications beyond the current problem, e.g., it allows us to show the
asymptotic Gaussianity and consistency of $k$-fold cross-validation for the
risk with dependent data (Appendix \ref{sec:k-fold-cv}).

We show (Section \ref{sec:ar_model_application}) that our conditions on the
model hold for linear autoregressive (AR($d$)) models estimated by least
squares --- again, whether or not the actual data-generating process is a
linear autoregression.  Simulation studies (Section \ref{sec:simulations})
indicate that convergence to the asymptotic coverage levels is quite rapid,
even when the model is highly mis-specified.

\subsection{Related work}
If the data source were not just stationary but independent and identically
distributed (IID), the definition of $R(\hat{f}_t)$ would reduce to the usual
$\Expect{L(Y_{t+1}, \hat{f}_t)}$, and we would be on familiar ground.  If we
want a point estimate of the risk, we might adjust the in-sample performance
through the analytical approximations of the various information criteria
\citep{Claeskens-Hjort-model-selection}, or we might turn to computational
procedures which attempt to directly simulate extrapolating to new observations
from the same distribution, such as cross-validation
\citep{Arlot-Celisse-on-cv} or the bootstrap
\citep{Efron-bootstrap-error-rate}\footnote{\citet{Bunke-Droge-bootstrap-vs-CV}
compares the properties of cross-validation and a bootstrapped generalization
error, but in a theoretical study involving IID linear regression with Gaussian
errors.  \citet{Shao-bootstrap-model-selection} studies the theoretical
properties of the bootstrapped generalization error for model selection
consistency in the linear regression, generalized linear regression, and
autoregressive model regimes.  He considers a residual resampling scheme and
shows that this bootstrap is not consistent for model selection, but an $m$ out
of $n$ variant of the bootstrap is consistent.}.  If we want an interval
estimate, there are again abundant analytical results from statistical learning
theory, based on notions of Vapnik-Chervonenkis dimension, Rademacher
complexity, etc.  \citep{Vidyasagar-on-learning-and-generalization,
Mohri-Rostamizadeh-Talwalkar}.  Many of these approaches, however, require at
the very least extensive re-thinking when applied to time series.

Thus, for example, while information criteria have been extensively studied for
selecting the order of time series models (mostly AR and ARMA models) since at
least the work of \citet{Shibata-order-selection, Hannan-order-estimation}, it
is well-known that their penalty terms they add to the in-sample risk are only
accurate under correct model specification
\citep{Claeskens-Hjort-model-selection}.  Similarly, while forms of
cross-validation for time series exist
\citep{Burman-Chow-Nolan-cross-validation-for-dependent,
Hardle-View-kernel-reg-for-time-series, Racine-feasible-cv,
Racine-consistent-cv-for-dependent-data}, we are unaware of work which studies
their properties under model mis-specification.  Both information criteria and
cross-validation moreover only provide point estimates of the risk,
rather than probabilistic bounds.

Extensive work has been done on extending the statistical-learning approach to
deriving such ``probably approximately correct'' bounds to time series
\citep{Meir-nonparametric-time-series, Karandikar-Vidyasagar-rates-of-UCEM,
Mohri-Rostamizdaeh-stability-bounds, Mohri-Rostamizadeh-rademacher-for-non-iid,
Kuznetsov-Mohri-non-stationary-mixing, risk-bounds-for-state-space-models}, but
this strand of work has, from our point of view, two drawbacks.  The first is
that the bounds derived hold uniformly over very wide ranges of processes.
This is obviously advantageous when one knows little about the data source, but
means that the bounds are often extremely conservative.  The second drawback is
that bounds often involve difficult-to-calculate quantities, especially the
beta-mixing coefficients which quantify the decay of correlations in the
stochastic process (see Definition \ref{defn:beta-mixing}).  These are
typically unknown, though not inestimable
\citep{estimating-beta-mixing-bernoulli}.  In contrast, our bounds, while
only asymptotically valid, are distribution-dependent and fully calculable,
though we do need to make a (weak) assumption about the mixing coefficients.

 \section{Proposed Estimator}
 In this section, we describe our algorithm in greater detail.  We begin by
discussing the resampling procedure, which will be used to generate both the
training and the test sets.  The circular block bootstrap variant that we will
use can be found in \citet{Lahiri-resampling-for-dependent}.

\subsection{The Circular Block Bootstrap (CBB)}

Suppose the functional of interest is a function of a $d$-dimensional
distribution function. For each observation $\{1, \ldots t \}$, generate a
“chunk” given by $ Z_i = (Y_i, \ldots, Y_{i+d -1})$. If an index $k$ exceeds
$t$, change the value of the index to $k \ \text{mod} \ t$.  By wrapping the
data around a circle, we ensure that each point is equally likely to be
resampled.\footnote{Had we ignored chunks with indices exceeding $t$, we would
  have had a moving-blocks bootstrap, which has the undesirable property that
  the bootstrap expected value no longer equals the sample mean.}

To capture the dependence structure in the data, we will need to resample
contiguous blocks of the chunks, given by
$\mathcal{B}_i =(Z_i, \ldots, Z_{i+\ell-1})$, where $\ell$ is the block-length.
Suppose the desired resample size is $N = t$.  Then we will resample
$ b = \ceil*{(N-d)/\ell}$ blocks, where the blocks are uniformly chosen from
$\{\mathcal{B}_1, \ldots, \mathcal{B}_t \}$.  Let $\{s(1),\ldots, s(b) \}$
denote the indices corresponding to the selected blocks.  The result of the
bootstrap procedure is the vector given by
$(\mathcal{B}_{s(1)}, \ldots,\mathcal{B}_{s(b)}) = (Z_{s(1)},Z_{s(1)+1},\ldots
, Z_{s(1)+ \ell - 1}, \ldots Z_{s(b) + \ell - 1})$. If the length exceeds
$N-d$, truncate entries at the end of the vector as needed. We can expand each
$Z_i$ to form the $ (N-d) \times d $ data matrix $\mathbf{X}^*$ given by:
\begin{equation}
  \mathbf{X}^* = \begin{bmatrix}
    Y_{s(1)} & \dots & Y_{s(1)+d-1} \\
    Y_{s(1)+1} & \dots & Y_{s(1)+d} \\
    \vdots & \ddots & \vdots \\
    Y_{s(b) + \ell - 1} & \dots & Y_{s(b) + \ell + d}
  \end{bmatrix}
\end{equation}
Note that it is also possible to perform a CBB procedure directly on
$Y_1, \ldots, Y_t$ to generate a bootstrapped series $Y_1^*, \ldots Y_N^*$, and
then construct the bootstrapped data matrix.  This approach, however, suffers
from the fact that some rows in bootstrapped data matrix contain observations
from two different blocks, leading to worse finite sample performance.

\subsection{Computing the Generalization Error Bound}

Our proposed procedure is described in Algorithm 1, which follows \citet[sec.\
8.3]{McDonald-thesis}.  Intuitively, we use the CBB to generate a training and
test set. We then compute an estimator of the generalization error, and take the
$(1-\alpha)$ quantile of this estimator across all bootstrap iterations and use
this to a construct confidence interval for the risk.

\alglanguage{pseudocode}
\begin{algorithm}[H]
\small
\caption{Bootstrap Bound on Prediction Error}
\begin{algorithmic}[1]
    \State Fit $\hat{f_t}$ on the time series $Y_1, \ldots Y_t$ and calculate the training error $\hat{R}(\hat{f_t})$
    \For {$i = 1 \to B$}
    \State Draw blocks with replacement to generate a training data matrix \\ \indent \indent $\mathbf{X}_{train}^* \in \R^{t \times d}$ and a test data matrix $\mathbf{X}_{test}^* \in \R^{t \times d}$
    \State Fit $\hat{f_t}$ on $\mathbf{X}_{train}^*$ and calculate the training error $\hat{R}(\hat{f_t})$.
    \State Calculate the test error $\widetilde{R}_t (\hat{f_t})$ on $\mathbf{X}_{test}^*$.
    \State Store the generalization term $\widehat{\eta}(\hat{f}) = \widetilde{R}_t(\hat{f_t}) - \hat{R}(\hat{f_t})$.
    \EndFor
    \State Find the $1 - \alpha$ percentile of $\widehat{\eta}(\hat{f_t})$. Add this to $\hat{R}(\hat{f_t})$.
\Statex
\end{algorithmic}
\label{the-algorithm}
\end{algorithm}

\subsection{Choosing the Block Length}

The theoretical results in the next section are compatible with a range of
block-length sequences, and so allow for data-driven block-length selection.
To our knowledge, there are no procedures in the literature that provide a
consistent estimate of the {\em optimal} block length for quantile estimation
of a general Hadamard-differentiable function.  (The methods of
\citet{Hall-Horowitz-Jing-bootstrap-block-size-rules} and
\citet{Lahiri-Furukawa-Lee-bootstrap-block-size-rule} only apply to functionals
in the ``smooth function of the mean'' framework.)  For convenience, we have
used the procedure of \citet{PolitisWhite2004}, which is optimized for the
sample mean but performs adequately in our simulation study\footnote{In
  additional simulations, not shown here, we used block-lengths of the form
  $C n^{1/3}$ for a large grid of values of $C$.  The procedure of
  \citet{PolitisWhite2004} was at least comparable, in terms of coverage, to
  the best-in-retrospect $C$.}.

\label{sec:method}

 \section{Main Results}
 \subsection{Preliminaries}
 In this section, we will prove the main theorem, which establishes sufficient
conditions for bootstrap consistency of an estimator of the generalization
error to hold.  Let
$\widetilde{R}_{N}(\hat{f}_t) \equiv \frac{1}{t_0} \sum_{i=t_0+1}^{N} L(y_i,
\hat{f}_t)$ denote the test error.  We will consider the case where $N=2t_0$. The sampling distribution
that we would like to approximate with our bootstrap estimator is that of
$\widehat{\eta}(\hat{f}_t) \equiv \widetilde{R}_{N}(\hat{f}_t) -
\widehat{R}_{t}(\hat{f}_t)$.


To derive limiting distributions, we will need assumptions about the dependence
structure.  We will assume that the process has a $\beta$-mixing coefficient
that decays at least at a cubic rate. While there are several equivalent
definitions of the $\beta$-mixing coefficient, below we state the version we
find to be most intuitive for strictly stationary processes.

\begin{definition}[$\beta$-mixing coefficient for stationary processes]
  \label{defn:beta-mixing}
  Let $Y$ be a stationary stochastic process and $(\Omega, \cF, P_\infty)$ be
  the probability space induced by the Kolmogorov extension. The coefficient of
  absolute regularity, or $\beta$-mixing coefficient $\beta_Y(k)$ is given by:
  \begin{align}
    \beta_Y(k) = \|P_{-\infty:0} \otimes P_{k:\infty}  - P_{-\infty:0} P_{k:\infty} \|_{TV}
  \end{align}
  where $\|\cdot \|_{TV}$ is the total variation norm,
  $P_{-\infty:0} \otimes P_{k:\infty}$ is the joint distribution of the blocks
  $\{Y_i\}_{i=-\infty}^0$ and $\{Y_i\}_{i=k}^\infty$ and
  $P_{-\infty:0} P_{k:\infty}$ is the product measure between the two blocks.
  We say that a process is $\beta$-mixing if $\beta_Y(k) \rightarrow 0$.
\end{definition}

From the definition, we can see that $\beta_Y \rightarrow 0$ implies asymptotic
independence as the gap between the past and future of the process grows.

We will also need to impose a differentiability condition on the estimated
parameters $\hat{\theta}(F)$, which take as input a $d$-dimensional
distribution function. Since the input is an infinite-dimensional object, we
will need a generalized notion of derivative.  The notion we use is tangential
Hadamard differentiability, which we define below:

\begin{definition}[Tangential Hadamard derivative]
  A map $\phi: \mathbb{D}_\phi \subset \mathbb{D} \mapsto \mathbb{E}$ is
  Hadamard differentiable at $\theta \in \mathbb{D}$, tangential to a set
  $\mathbb{D}_0 \subset \mathbb{D}$ if there exists a continuous linear map
  $\phi_\theta': \mathbb{D} \mapsto \mathbb{E}$ such that:
  \begin{align}
    \frac{\phi(\theta + t_n h_n) - \phi(\theta)}{t_n} \rightarrow \phi_\theta'
  \end{align}
  as $n \rightarrow \infty$ for all converging sequences $t_n \rightarrow 0$
  and $h_n \rightarrow h \in \mathbb{D}_0$.
\end{definition}

Finally, let $\ell^\infty(A,B)$ denote the space of bounded functions mapping
from $A$ to $B$.  If the second argument is omitted, take $B$ to be
$\mathbb{R}$. Now we are ready to state the main theorem.

\begin{theorem}
  Let $Y_1, \ldots, Y_t$ be observations from a strictly stationary
  $\beta$-mixing process $Y$ supported on $[-M,M]$ for some $M \in \R$.
  Suppose our loss function can be expressed as
  $L(z, \theta) \in \ell^\infty( \mathcal{Z} \times \Theta, \R)$,
  where $z \in \mathcal{Z} = [-M,M]^d$ and
  $\theta \in \Theta$, where $\Theta$ is a subset of a normed space.
  Assume that $X_{test}^*$ and $X_{train}^*$ are generated by the CBB procedure
  described in Section 2 with the block size $b(t) \rightarrow \infty$.
  Further, let $\hat{\theta}(\cdot): \mathcal{A} \subset \ell^\infty(\mathcal{Z}, \mathbb{R}) \mapsto \Theta$, where $F \in \mathcal{A}$ and $\theta_0 = \hat{\theta}(F)$.  Also assume the following conditions:

\begin{enumerate}
  \renewcommand{\labelenumi}{\theenumi}
  \renewcommand{\theenumi}{C\arabic{enumi}}
\item \label{mixing condition} $\sum_{k=1}^\infty (k+1)^2 \beta_Y(k) < \infty$
\item \label{block condition}
  $\lim \sup_{t \rightarrow \infty} t^{-\frac{1}{2}}b(t) < \infty$ 
\item \label{loss condition} $L(\cdot,\theta)$ is appropriately tangentially Hadamard differentiable with respect to
  $\theta$ at $\theta_0$ and $L(z,\theta_0)$ is differentiable with respect to $z$ with
  bounded mixed partial derivatives up to the $d$th order on $\mathcal{Z}$
\item \label{estimator condition}$\hat{\theta}(F)$ is appropriately
  tangentially Hadamard differentiable at $F$
\end{enumerate}
Then, the bootstrap is consistent for $\eta^*(\hat{f}_t^*)$.
\end{theorem}
 \label{Main Theorem Statement}
 \subsection{Proof of Main Theorem}
 \label{Proof of Main Theorem}
 We will show bootstrap consistency using the functional delta method for the
bootstrap.  We will define additional notions needed to state this theorem
below. The definitions we adopt here can be found in \citet{Kosorok-intro}.

Bootstrap consistency is known to follow from a certain conditional weak
convergence, which we will define below.  For ease of exposition, we do not address measurability issues in our definitions; again, the reader is referred to \citet{van-der-Vaart-Wellner-weak-conv} or \citet{Kosorok-intro} for details.

\begin{definition}
  We say that $X_n$ converges weakly in probability conditional on the data, or
  $X_n \overset{P}{\underset{W} \rightsquigarrow} X$, if
  $\sup_{h \in BL_1} |\DataExpect{h(X_n)} - \Expect{h(X)}|
  \xrightarrow{P} 0$
where $BL_1$ is the space of bounded functions with Lipschitz
 norm $\leq 1$,   and $\DataExpect{\cdot}$
  is conditional expectation over the weights $M$ given the data.
\end{definition}

\begin{proposition}[Functional Delta Method for the Bootstrap, \citet{Kosorok-intro}, Theorem 12.1]
  \label{FDM}
  For normed spaces $\mathbb{D}$ and $\mathbb{E}$ , let
  $\phi: \mathbb{D}_\phi \subset \mathbb{D} \rightarrow \mathbb{E}$ be a
  Hadamard differentiable map at $\mu$, tangential to
  $\mathbb{D}_0 \subset \mathbb{D}$, with derivative $\phi_\mu'$.  Let
  $\mathbb{X}_n$ and $\mathbb{X}_n^*$ have values in $\mathbb{D}_\phi$, with
  $r_n (\mathbb{X}_n - \mu) \rightsquigarrow \mathbb{X}$, where $\mathbb{X}$ is
  tight and takes values in $\mathbb{D}_0$ for some sequence of constants
  $0 < r_n \rightarrow \infty$, the maps $W_n \rightarrow h(\mathbb{X}_n)$ are
  measurable for every $h \in C_b(\mathbb{D})$ almost surely, and where
  $r_n c (\mathbb{X}_n^* - \mathbb{X}_n) \overset{P}{\underset{W}
    \rightsquigarrow} \mathbb{X}$ for a constant $0 < c < \infty$.  Then
  $r_n c (\phi(\mathbb{X}_n^*) - \phi(\mathbb{X}_n)) \overset{P}{\underset{W}
    \rightsquigarrow} \phi_{\mu}'(\mathbb{X})$.
\end{proposition}

As discussed in \citet{Lahiri-resampling-for-dependent}, we will construct
chunks of size $d$ such that $Z_i = (Y_i, \ldots, Y_{i+d+1})$ and consider
functionals of the distribution function of $Z$. Let $t_0=t-d+1$.
In addition, let $P_{t}= \frac{1}{t_0} \sum_{i=1}^{t_0} \delta_{Z_i}$ denote
the empirical measure for the training data and
$P_{N} = \frac{1}{t_0} \sum_{i= t_0 +1}^{N} \delta_{Z_i}$ denote the
empirical measure for the test data, where $N = 2t_0$. Similarly, let $P_t^*$ and $P_N^*$ denote
the bootstrap measure for the training set and the test set, respectively,
where each measure is conditioned on $t$ points.

For $z \in \mathcal{Z}$, let the $\leq$ operator be defined element-wise: that
is, $a \leq b$ iff $a_i \leq b_i$ for all
$i \in \{1, \ldots, d\}$.  If we let $\cF$ be the function class given by:
\begin{align}
\label{function class}
\cF = \{f_s: s \in \mathcal{Z}, f_s(z) = \cI(z \leq s ) \}
\end{align}
It follows that the empirical distribution function for the training set and
test set can be expressed point-wise as $F_t(s) = P_t f_s$ and
$F_N(s) =P_N f_s$, respectively.  Analogous expressions hold for
bootstrap measures.  To represent the entire distribution function, we can view
$F_t(\cdot)$ and $F_N(\cdot)$ as elements of
$\ell^\infty(\Omega \times \cF, \R)$ and $F_t^*(\cdot)$ and $F_N^*(\cdot)$ as
elements of $\ell^\infty(\Omega \times \bar{\Omega} \times \cF, \R)$, where $\bar{\Omega}$ is the probability space associated with the bootstrap weights.  For a
fixed sample path, we may view these mappings as belonging to
$\ell^\infty(\cF, \R)$.

Our first step is to translate $\hat{R}(\hat{f_t})$ and
$\widetilde{R}_{N}(\hat{f_t})$ into plug-in estimators. This is done in the
following lemma:
\begin{lemma}
  \label{plugin lemma}
  Let $A$, $B \in \ell^\infty(\cF, \R) $ and define the mappings:
  \begin{eqnarray}
    \label{train functional}
    \phi_{train}(A,B) &= & \int L(z ,\hat{\theta}(A)) \ dA \\
    \label{test functional}
    \phi_{test}(A,B) &= & \int L(z ,\hat{\theta}(A)) \ dB \\
    \label{risk functional}
    \phi_{risk,F}(A) &= & \int L(z ,\hat{\theta}(A)) \ dF
  \end{eqnarray}
  Then, $\hat{R}(\hat{f}_t)$ and $\widetilde{R}_{N,}(\hat{f}_t)$ are plug-in
  estimators for $\phi_{train}(\cdot,\cdot)$ and $\phi_{test}(\cdot,\cdot)$,
  respectively. Furthermore, suppose that $Y$ is a stationary $\beta$-mixing
  process satisfying $\sum_{k=1}^\infty \beta_Y(k) < \infty$ and $\hat{f}_t$ is
  a function of $\theta(F_t)$ and $z_t = (Y_{t}, \ldots Y_{t-d})$ such that
  $L(Y_{t+1}, \hat{f}_t) = L(z, \theta(F_t))$.  Then for each
  $t \in \N$:
  \begin{align}
    \label{stationary risk}
    \lim_{m \rightarrow \infty} \frac{1}{m} \sum_{i=1}^m \Expect{L(Y_{t+i}, \hat{f}_t) | Y_1, \ldots , Y_t} = \int{L(z ,\hat{\theta}(F_t)) dF} = \phi_{risk,F}(F_t)
  \end{align}

  \begin{prooflemma}
    The fact that the training and test errors are plug-in estimators of
    functionals given in (\ref{train functional}) and (\ref{test functional})
    follows immediately by inspection.  To see (\ref{stationary risk}),
    consider the measurable space $(\Omega^{t+d}, \mathcal{A}^{t+d})$, where
    superscripts denote product spaces, and let $\{\mu_m\}_{m \in \mathbb{N}}$
    be a sequence of probability measures corresponding to the distribution of
    $(Y_1, \ldots, Y_t, Y_{t+m}, Y_{t+m+d-1})$.  By the definition of
    $\beta_Y(\cdot)$, it follows that:
    \begin{align}
      \|\mu_m - \mu_\infty \|_{TV} \leq \beta_Y(m)
    \end{align}
    where $\mu_\infty$ is the product measure between $(Y_1, \ldots Y_t)$ and
    $(Y_{t+1}, \ldots Y_{t+d-1})$.  Now we will use a result from
    \citet{Schervish-Seidenfeld-approach-to-consensus}, stated in a less general form for current purposes:
    \begin{lemma}[\citet{Schervish-Seidenfeld-approach-to-consensus}]
      Let $Q$ and $R$ be probability distributions on
      $(\Omega^j, \mathcal{A}^j)$ where $j$ is allowed to be $\infty$.  Define
      the space of histories $(H_t, \mathcal{H}_t)$, where $H_t = \Omega^t$ and
      $\mathcal{H}_t = \mathcal{A}^t$ and let $h_t \in H_t$ be a
      history. Suppose $Q$ and $R$ permit regular conditional probabilities,
      denoted $Q(\cdot|h_t)$ and $R(\cdot|h_t)$, respectively. Then,
      \begin{align}
        \|Q - R \|_{TV} < ab \implies Q( \|Q(\cdot| h_t) - R(\cdot|h_t) \|_{TV} > a) < b
      \end{align}
    \end{lemma}
    Applying this lemma with the assumed mixing conditions yields:
    \begin{align}
      \sum_{m=1}^\infty P_\infty \left(\|\mu_m(\cdot| h_t) - \mu_\infty(\cdot|h_t) \|_{TV} > a_m \right) <\infty
    \end{align}
    where $P_\infty$ is the distribution of the process $Y$ and
    $a_m \rightarrow 0$ \footnote{The fact that such an $a_m$ exists can be
      shown using the Open Mapping Theorem in functional analysis.  See for
      example, \citet[Problem 25.23]{Driver-analysis-tools}.}, which implies
    that $\|(\mu_m(\cdot | h_t) - \mu_\infty(\cdot|h_t)\|_{TV} \rightarrow 0$,
    $P_\infty-a.s.$ by Borel-Cantelli Lemma. Now, we will use the fact that for
    all bounded $f$,
    $\|(\mu_m(\cdot | h_t) - \mu_\infty(\cdot | h_t)\|_{TV}\rightarrow 0$
    implies that:
    \begin{align}
      \int{f d\mu_m(\cdot | h_t)} \rightarrow \int{f d\mu_\infty(\cdot | h_t)}
    \end{align}
    Applying this result to $L(z,\theta)$,which we assumed to be bounded, we
    conclude that
    $\lim_{m \rightarrow \infty}{
      \Expect{L(Z_{t+m},\hat{\theta}(F_t)|Y_1,\ldots Y_t)}} =
    \int{L(z,\theta(F_t) dF}$ $P_\infty-a.s.$, which implies that the Cesaro
    mean also converges to the desired limit.
  \end{prooflemma}
\end{lemma}

\begin{remark}
  Going back to at least \citet{Akaike-predictor-identification}, a proposal
  for risk in the time series setting is $\int L(z,\theta(F_t) dF$, where $Z$
  and $Y_1, \ldots Y_t$ are observations from different runs of the process
  (and are therefore independent).  While easier to work with than many
  proposals for the conditional risk, it suffers from the fact that we are
  often interested in how our model will perform on future data from the same
  realization of the process.  Our result here provides sufficient conditions
  for which this convenient notion of risk is equivalent to that of
  \citet{CRS-Leo-predictive-mixtures}.
\end{remark}

\begin{remark}
  Summability of $\beta$-mixing coefficients ensures convergence of the shifted
  risk to the desired limit, which is stronger than what would be minimally
  needed to show Cesaro convergence.  One could make a similar argument based
  on a uniform bound on the decay of correlations in Cesaro-mean, but
  $\beta$-mixing is used in the main theorem, so we will not introduce
  nonstandard notions of dependence here.
\end{remark}

Now define
$\mathbb{H}_{t}, \mathbb{H}^*_{t} \in \ell^\infty(\cF, \R^2)$,
which represent centered bivariate empirical processes related to the empirical
distribution functions and bootstrapped empirical distribution function,
respectively:
\begin{align}
\label{H process}
\mathbb{H}_{t}(f_s) &=
\begin{pmatrix}
 \mathbb{G}_t(f_s) \\
 \mathbb{G}_{N}(f_s)
 \end{pmatrix}
 =
\sqrt{t_0} \begin{bmatrix}
    P_t \ f_s - P \ f_s  \ \\
  P_{N} \ f_s - P \ f_s
 \end{bmatrix}
%
 \end{align}

 \begin{align}
 \label{bootstrap H process}
 \mathbb{H}^*_{t}(f_s) &=
\begin{pmatrix}
 \mathbb{G}_t^* (f_s) \\
 \mathbb{G}_{N}^* (f_s)
 \end{pmatrix}
=
 \sqrt{t_0}
\begin{bmatrix}
P_t^* \ f_s - P_t \ f_s  \\
 P_N^* \ f_s - P_t \ f_s
 \end{bmatrix}
\end{align}
To apply the functional delta method, we need to show that $\mathbb{H}_{t}$ and  $\mathbb{H}_{t}^*$ converge to the same Gaussian Process, up to a multiplicative constant on the covariance function.  We will derive the limiting distribution of $\mathbb{H}_t$ in the lemma below.
\begin{lemma}
\label{joint convergence lemma}
Let $Y$ be a $\beta$-mixing process with mixing rate satisfying \ref{mixing condition} and consider the statistic $\mathbb{H}_t$ defined in (\ref{H process}). Then,
\begin{align}
\label{limit process}
\mathbb{H}_{t} \rightsquigarrow \mathbb{H} \equiv  \mathbb{G} \times \mathbb{G}
\end{align}
where $\mathbb{H}$ is a bivariate Gaussian process, with the $\times$ symbol denoting independence. Furthermore, is a mean zero Gaussian Processes with covariance structure given by:
\begin{align}
\label{covariance function}
\Gamma(f,g) = \lim_{k \rightarrow \infty} \sum_{i=1}^\infty \text{Cov} \ (f(Z_k), g(Z_i)) \  \forall f,g \in \cF
\end{align}
\end{lemma}

\begin{prooflemma}
  We will prove joint convergence using the following result from the
  literature:

\begin{proposition}[\citet{van-der-Vaart-Wellner-weak-conv}, 1.5.3]
  Let $X_\alpha: \Omega_\alpha \mapsto \ell^\infty(U)$ and
  $Y_\alpha: \Omega_\alpha \mapsto \ell^\infty(V)$ be asymptotically tight
  nets such that
  \begin{align}
    \left( X_\alpha(u_1), \ldots X_\alpha(u_k), Y_\alpha(v_1) , \ldots , Y_\alpha(v_l)  \right)
    \rightsquigarrow \left( X(u_1), \ldots X(u_k), Y(v_1) , \ldots , Y(v_l)  \right)
  \end{align}
  for stochastic processes $X$ and $Y$. Then there exists versions of $X$ and
  $Y$ with bounded sample paths and
  $(X_\alpha, Y_\alpha) \rightsquigarrow (X,Y)$ in
  $\ell^\infty(U) \times \ell^\infty(V)$.
\end{proposition}

That is, to show weak convergence of a vector-valued stochastic process, it
suffices to show (a) marginal asymptotic tightness and (b) finite-dimensional
convergence of arbitrary joint distributions.  We will set $U = V =\cF$, where
$\cF$ is the function class corresponding to indicator functions defined in
(\ref{function class}).  To show (a), it is sufficient to prove that there
exists a semimetric $\rho$ for which $U$ is totally bounded and
\begin{align}
\label{stochastic equicontinuity}
\lim_{\delta \downarrow 0} \lim_{t \rightarrow \infty} P^{out} \left\{\sup_{u,v \in U \text{with} \rho(u,v) <\delta} |X_t(u) - X_t(v)| > \epsilon \right\} = 0 \ \forall \epsilon > 0
\end{align}
Generally, proving this condition is nontrivial.  However, it is known to hold
for the process $\{\sqrt{t} (1/t \sum_{i=1}^t f(X_i) - \Expect{f(X)}\}_\cF$
under some additional conditions:

\begin{proposition}[\citet{Kosorok-intro}, Theorem 11.25, after \citet{Arcones-Yu-CLT}, Theorem 2.1]
  \label{CLT VC}
  Let $Y$ be a stationary sequence in a Polish space with marginal distribution
  P, and let $\cF$ be a class of functions in $L_2(P)$.  Let
  $\mathbb{G}_n(f) = P_nf - Pf$. Suppose there exists a $2 < p < \infty$ such
  that:
  \begin{enumerate}
  \item[(a)]
    $\lim_{k \rightarrow \infty} k^{2/(p-2)}(\log k)^{2(p-1)/(p-2)}\beta_Y(k)
    =0 $
  \item[(b)] $\cF$ is permissible, VC and has envelope F satisfying
    $P^{out}F^p < \infty$.
  \end{enumerate}
  Then $\mathbb{G}_n \rightsquigarrow \mathbb{G}$ in $\ell^\infty(\cF)$, where
  $\mathbb{G}$ is a mean 0 Gaussian Process with covariance structure given in
  (\ref{covariance function}).
\end{proposition}

We can apply the result to each process marginally to show that (a) holds. See
Appendix \ref{Arcones Yu Modification} for additional details.

To show (b), we will use the Cramer-Wold device. We will need a central limit
theorem for nonstationary triangular arrays satisfying some mixing condition to
show univariate convergence.  To this end, we will state a modification of a theorem of
\citet{Ekstrom-CLT-for-strong-mixing}\footnote{This paper uses the notion of
weakly-approaching random variables
\citep{Belyaev-Sjostedt-de-Luna-weakly-approach}, which does not require the
existence of a limiting distribution; this is more general than what we need
here.}, itself a modification of
\citet{Politis-Romano-Wolf-subsampling}, which requires a slightly different notion of mixing:

\begin{definition}[$\alpha$-mixing coefficient]
\begin{align}
\alpha_X (k) = \sup_{l \geq 1} \{ |P(AB) - P(A) P(B)| :A \in \sigma_{-\infty}^{l}, B \in \sigma_{l+k}^\infty \}
\end{align}
\end{definition}

It is a well-known result in mixing theory that $\alpha_X \leq \beta_X$
\citep[Eq.\ 1.1]{Bradley-strong-mixing}.

\begin{proposition}[\citet{Ekstrom-CLT-for-strong-mixing}, Theorem 1]
  \label{nonstationary CLT}
  Let $\{ X_{n,i}, 1 \leq i \leq d_n \}$ be a triangular array of random
  variables, with $\alpha_n(\cdot)$ being the $\alpha$-mixing coefficients for
  the $n^{\mathrm{th}}$ row, and
  $\bar{X}_{n} \equiv d_n^{-1} \sum_{i=1}^{d_n}{X_{n,i}}$ being the sample
  mean of that row.
  Assume the following. For some $\delta > 0$,
  \begin{enumerate}
    \renewcommand{\labelenumi}{\theenumi}
    \renewcommand{\theenumi}{B\arabic{enumi}}
  \item \label{nonstationary CLT moments condition}
    $\Expect{|X_{n,i} - \Expect{X_{n,i}}|^{2+\delta}} <c $ for some $c > 0$ and all $n,i$
  \item \label{nonstationary CLT mixing condition}
    $\sum_{k=0}^\infty {(k+1)^2 \alpha_n^{\frac{\delta}{4+\delta}}(k)} < c$ for
    all $n$
  \end{enumerate}
  Then, if
  $\sigma^2 = \lim_{n \rightarrow \infty}
  \text{Var}(\sqrt{d_n}\bar{X}_{n})$ exists,
  \begin{equation}
    d_n^{1/2} (\bar{X}_{n} - \Expect{\bar{X}_{n}}) \rightsquigarrow N(0, \sigma^2)
  \end{equation}
\end{proposition}

Let $\{u_1, \ldots, u_k$, $v_1, \ldots, v_l\}$ be a set of arbitrary elements
of $\mathcal{Z}$ for any $k,l \in \mathbb{N}$ and let
$\{\lambda_1, \ldots, \lambda_k,\gamma_1, \ldots, \gamma_l\}$ be arbitrary
elements of $\R$.  After applying Cramer-Wold, we are left with the following
sample mean:
\begin{equation}
S_{d_t} \equiv \frac{1}{d_t} \sum_{i =1}^{d_t}{g_i(Z_i)} = \frac{1}{t_0} \sum_{i=1}^{t_0}{\sum_{j=1}^{k}{\lambda_j \cI(Z_i \leq u_j)}} + \frac{1}{t_0} \sum_{i=t_0+1}^{N}{\sum_{j=1}^{l}{\gamma_j \cI(Z_i \leq v_j)}}
\end{equation}
We will view $\{g_i(Z_i)\}_{i=1}^{d_t}$ as elements of the triangular array.
\ref{nonstationary CLT moments condition} follows for our class $\cF$. Although
the process is now nonstationary, it still satisfies \ref{nonstationary CLT
  mixing condition} since the mixing coefficient of $\{g_i(Z_i)\}_{i=1}^{d_n}$
is bounded by $\alpha_Y(\cdot)$. Since we can take $\delta$ to be arbitrarily
large and $\alpha_Y(\cdot) \leq \beta_Y(\cdot)$, (b) follows.

Now we will need to check that the limiting variance matches that of limiting
Gaussian process defined in (\ref{covariance function}).  Simple calculation
reveals that:
\begin{align*}
\text{Var}(\sqrt{d_t} S_{d_t}) &= \frac{1}{d_t} \left[\sum_{i=1}^{t_0} \sum_{j=1}^{t_0} \text{Cov}(g_i(Z_i),g_j(Z_j)) + \sum_{i=t_0+1}^{N}\sum_{j=t_0+1}^{N} \text{Cov}(g_i(Z_i),g_j(Z_j)) \right. \\
 &+ \left.  \sum_{i=1}^{t_0} \sum_{j=t_0+1}^{N} \text{Cov}(g_i(Z_i),g_j(Z_j)) +  \sum_{i=t_0+1}^{N} \sum_{j=1}^{t_0} \text{Cov}(g_i(Z_i),g_j(Z_j)) \right]
\end{align*}
The desired result will follow if we can show that the last two terms on the RHS
goes to 0.  We find conditions for this to be true below.

\begin{lemma}
  Suppose $Y$ is a stationary stochastic process that is polynomially
  $\beta$-mixing and let $\mathcal{G}$ be the linear hull of $\cF$. That is,
  $\mathcal{G}$ is the function class corresponding to finite linear
  combinations of $\cF$, given by:
  \begin{align}
    \mathcal{G} = \left \{ g = \sum_{i=1}^m a_i f_i \ , \ a_i \in \R, f_i \in \cF, m \in \mathbb{N} \right\}
  \end{align}
  where $\cF$ is the function class given in (\ref{function class}).
  Then, for any
  $g_1, g_2 \in \mathcal{G}$:
  \begin{align}
    \lim_{t \rightarrow \infty} \Cov(\sqrt{t_0} P_t \ g_1, \sqrt{t_0} P_{N} \ g_2) = 0
  \end{align}
\end{lemma}
\begin{proof}
  First, recall \citet{Davydov-inequality}'s inequality, which states that:
  \begin{align}
    |\Cov(X,Y)| \leq 10 \beta(\sigma(X), \sigma(Y))^r \|X\|_p \|Y\|_q
  \end{align}
  where $r+1/p + 1/q = 1$.  For any $g \in \mathcal{G}$,
  $\|g \|_\infty \leq \sum_{i=1}^m |a_i|$. We can therefore set
  $r = 1 - \epsilon$ for an arbitrarily small $\epsilon > 0$.  Letting
  $k = |i -j|$ and applying the above inequality, we see that:
  \begin{align*}
    \left| \Cov(\sqrt{t_0} P_t g_1, \sqrt{t_0} P_{N} g_2) \right| &= \left| \Cov \left( \frac{1}{\sqrt{t_0}} \sum_{i=1}^{t_0} g_1(z_i),
                                                                    \frac{1}{\sqrt{t_0}}\sum_{j=t_0+1}^{N} g_2(z_j) \right) \right|
    \\ &\lesssim \frac{1}{t_0}  \sum_{i=1}^{t_0} \sum_{j=t_0+1}^{N} \ \beta(|i-j|)^r
\\ &\lesssim \frac{1}{t_0} \sum_{k=1}^{\ceil*{t_0/2}}  k \  \beta(k)^r \rightarrow 0
  \end{align*}
\end{proof}
Since we have shown both (a) and (b), we have the joint convergence of the
empirical process defined in (\ref{H process}) to a tight element in
$\ell^\infty(\cF) \times \ell^\infty(\cF)$.  Furthermore, since each of the
finite dimensional distributions are Gaussian, the limiting process is
Gaussian.  Asymptotic independence of the marginal Gaussian processes follows
from asymptotic independence of finite dimensional distributions; see
\citet[Lemma 11.1]{Kallenberg-mod-prob}.
\end{prooflemma}

\begin{remark}
When the stochastic process is compactly supported, this result holds for an arbitrary permissible VC class.  In a more general setting, there is a trade-off between assumptions on moments and mixing rates.
\end{remark}

\begin{remark}
For $d$-dimensional distribution functions, this lemma holds under $\alpha$-mixing due to a result of \citet[Theorem 7.3]{Rio-asymptotic-theory-weakly-dependent}.
    
\end{remark}

Now, finite VC dimension, along with the assumed mixing conditions, allows us
to invoke the following theorem due to
\citet{Radulovic-bootstrap-empirical-process-stationary}:

\begin{proposition}[\citet{Kosorok-intro}, Theorem 11.26, after \citet{Radulovic-bootstrap-empirical-process-stationary}, Theorem 1]
  \label{Bootstrap CLT}
  Let $Y$ be a stationary sequence taking values in $\R^d$ for some $d \in \N$
  with marginal distribution $P$, and let $\cF$ be a class of functions in
  $L_2(P)$. Let $\mathbb{G}^*_n(f) = P^*_n f - P_n f$ Also
  assume that $Y_1^*, \ldots, Y_n^*$ are generated by the CBB procedure
  with block size $b(n) \rightarrow \infty$ as $n \rightarrow \infty$, and that
  there exists $2 < p < \infty$, $q > p/(p-2)$, and $0 < \rho < (p-2)/[2(p-1)]$
  such that
  \begin{enumerate}
  \item[(a)] $\lim \sup_{k \rightarrow \infty} k^q \beta_Y(k) < \infty$
  \item[(b)] $\cF$ is permissible, VC, and has envelope $F$ satisfying
    $P^{out} F^P < \infty$, and
  \item[(c)] $\lim \sup_{n \rightarrow \infty} n^{- \rho} b(n) < \infty$.
  \end{enumerate}
  Then $\mathbb{G}_n^* \overset{P}{\underset{W} \rightsquigarrow}  \mathbb{G}$ in
  $\ell^\infty(\cF)$, where $\mathbb{G}$ is a mean 0 Gaussian Process with
  covariance structure given in (\ref{covariance function}).
\end{proposition}
Therefore, it follows that:
\begin{align}
 \mathbb{G}_{t}^* \overset{P}{\underset{W} \rightsquigarrow} \mathbb{G}
\end{align}
See Appendix \ref{Radulovic Modification} for details.

Since $\cF$ is permissible, by \citet{Yu-rates-of-convergence} the
measurability criteria are satisfied.

Finally, the functional delta method requires tangential Hadamard
differentiability of $\phi_{test}(A,B) - \phi_{train}(A,B) $ at our limit point
(F,F).  Note that $\phi_{test}(\cdot)$ can be expressed as the following
composition of mappings:
\begin{align}
\phi_{test} : (A,B) \xrightarrow{(a)} (B, \hat{\theta}(A)) \xrightarrow{(b)} (B, L(z,\hat{\theta}(A))) \xrightarrow{(c)} \int L(z,\hat{\theta}(A)) \ dB
\end{align}

Due to the Chain Rule, which we will state below, it will suffice to establish
tangential Hadamard differentiability for each of the intermediate mappings.
\begin{proposition}[Chain Rule, \citet{van-der-Vaart-Wellner-weak-conv}, Lemma 3.9.3]
\label{Chain Rule}
If $\phi: \mathbb{D}_\phi \subset \mathbb{D} \mapsto \mathbb{E}_\phi$ is
Hadamard differentiable at $\theta \in \mathbb{D}_\psi$ tangentially to
$\mathbb{D}_0$ and $\psi: \mathbb{E}_\psi \mapsto \mathbb{F}$ is Hadamard
differentiable at $\phi(\theta)$ tangentially to $\phi_\theta'(\mathbb{D}_0)$,
then $\psi \circ \phi : \mathbb{D}_\phi \mapsto \mathbb{F}$ is Hadamard
differentiable at $\theta$ tangentially to $\mathbb{D}_0$ with derivative
$\psi'_{\phi(\theta)} \circ \phi_\theta'$.
\end{proposition}

Differentiability of (a) follows from \ref{estimator condition}, (b) from
\ref{loss condition}, and (c) follows from a multivariate generalization of the
following proposition (see Appendix \ref{van der Vaart Wellner Modification}
for details):

\begin{proposition}[Integration,  \citet{van-der-Vaart-Wellner-weak-conv}, Lemma 3.9.17]
  \label{Integration}
  Given a cadlag function $A$ and a function of bounded variation $B$ on an
  interval $[a,b] \in \bar{\R}$, define
  \begin{align}
    \phi(A,B) = \int_{[a,b]} A \ dB
  \end{align}
  Then, $\phi: D[a,b] \times BV_M[a,b] \mapsto \R$ is Hadamard
  differentiable at each $(A,B) \in D_\phi$ such that $\int |dA| < \infty$.
  The derivative is given by:
  \begin{align}
    \phi_{A,B}^\prime(\alpha, \beta) = \int_{[a,b]} A \ d \beta + \int_{[a,b]} \alpha \ dB
  \end{align}
\end{proposition}
Condition \ref{loss condition} is stronger than what is needed to establish
Hadamard differentiability for step (b), but ensures that the derivative is
well-defined.  See Appendix \ref{van der Vaart Wellner Modification} for
details.

Similar reasoning holds for $\phi_{train}(\cdot)$ and Hadamard
differentiability of the difference holds due to the fact that the derivative
is a linear operator.  Then by the Functional Delta Method for the bootstrap,
the bootstrap is consistent for $\eta^*(\hat{f}_t)$.

\begin{remark}
  Adding an estimate of the $1-\alpha$ quantile of the generalization error to
  the training error is not the only way to construct confidence intervals.  In
  particular, one could also bootstrap the test error.
\end{remark}

 \subsection{Application of main theorem to $AR(p)$ models under weak assumptions}
 The motivating example for this theorem is the setting in which we use an $AR(p)$ model to predict a stationary $\beta$-mixing process.  We will state a corollary of our theorem that establishes bootstrap consistency of our procedure in this setting under squared error loss.

\begin{corollary}
Let $Y$ be a compactly-supported stationary stochastic process satisfying the mixing condition \ref{mixing condition} and $\Theta$ be a compact subset of $\mathbb{R}^p$ for some $p \in \mathbb{N}$.  Suppose $L(z,\theta) = [ z^T (-\theta,1)]^2$, where $\hat{\theta}(F)$ is the least square estimator for an $AR(p)$ model.  Further suppose that $b(n)$ satisfies \ref{block condition}.
Then the bootstrap is consistent for $\eta^*(\hat{f}_t^*)$.
\end{corollary}

\begin{proof}
We will start by establishing \ref{loss condition}, appropriate tangential Hadamard differentiability of $L: \Theta \mapsto \ell^\infty(\mathcal{Z}, \mathbb{R})$ at $\theta_0$.  For notational convenience, let $z = (x,y)$, where $y = y_d$ and $x = (y_1, \ldots y_{d-1})$.  We will establish the required tangential Hadamard differentiability by showing that:
\begin{align}
\label{squared error hadamard derivative}
\norm{\frac{L(\cdot, \theta_0 + t_n h_n) -  L(\cdot, \theta_0)}{t_n} - L_{\theta_0}'(\cdot, h)}_\infty \rightarrow 0
\end{align}
where: 
\begin{align}
L_{\theta_0}'((x,y), h) = 2(y - \theta_0^T x) h^T x 
\end{align}

The expression on the LHS of (\ref{squared error hadamard derivative}) may be rewritten as:
\begin{align}
\label{squared error decomposition}
\norm{2 (y - x^T \theta_0)x^T(h_n - h) }_\infty + |t_n| \cdot \norm{ (x^T h_n)^2}_\infty 
\end{align}

For the first term in  (\ref{squared error decomposition}), there exists $C < \infty$ such that for all $h_n$:
\begin{align}
\max_{x,y}|2 (y - x^T \theta_0) x^T(h_n - h)| < C||h_n - h||_1
\end{align}
 Since $||h_n - h||_\infty \rightarrow 0$, this term goes to $0$. Since $\max_x (x^T h_n)^2 \rightarrow M^2 ||h||_1^2$ and $t_n \rightarrow 0$, the second term also converges to 0, and Hadamard differentiability at $\theta_0$ follows.  
 


The only condition that remains to be checked is \ref{estimator condition}, the tangential Hadamard differentiability of the least squares estimator.

We will start by introducing some notation.  Given a subset $\Theta$ of a
Banach space and another Banach space $\mathbb{L}$, let
$\ell^\infty(\Theta, \mathbb{L})$ be the Banach space of uniformly bounded
functions $z: \Theta \mapsto \mathbb{L}$.  Let $Z(\Theta, \mathbb{L})$ be
the subset consisting of all maps with at least one zero.  Let
$\phi: Z(\Theta, \mathbb{L}) \mapsto \Theta$ be a map that assigns one of
its zeros to each element $z \in Z(\theta, \mathbb{L})$. In the squared error
case, $\Theta$ represents the domain of the coefficients and $\mathbb{L}$
represents the codomain of the estimating equations.  Further let lin $\Theta$
denote the linear closure of $\Theta$.  Hadamard differentiability is
established by the following:

\begin{proposition}[Z-Estimators,  \citet{van-der-Vaart-Wellner-weak-conv}, Lemma 3.9.34]
\label{Z estimator}
Assume $\Phi: \Theta \mapsto \mathbb{L}$ is uniformly norm-bounded, is one-to-one, possesses a $\theta_0$ and has an inverse (with domain $\Phi(\Theta)$) that is continuous at 0.  Let $\Phi$ be Frechet differentiable at $\theta_0$ with derivative $\dot{\Phi}$, which is one-to-one and continuously invertible on lin $\Theta$. Then the map $\phi: Z(\Theta, \mathbb{L}) \subset \ell^\infty(\Theta, \mathbb{L}) \mapsto \Theta$ is Hadamard differentiable at $\Phi$ tangentially to the set of $z \in \ell^\infty(\Theta, \mathbb{L})$ that are continuous at $\theta_0$. The derivative is given by $\phi_\Phi(z) = - \dot{\Phi}_{\theta_0}^{-1}(z(\theta_0))$.
\end{proposition}

This result establishes Hadamard differentiability of the least squares
estimator as a function of $\Phi(\cdot)$.  We can view $\Phi_n(\cdot)$ as a
Hadamard differentiable mapping of distribution functions to complete the
proof.  We will show this last step in the proposition below:

\begin{proposition}
    Let $\Upsilon: \mathcal{A}  \subset  \ell^\infty(\mathcal{Z},\mathbb{R})  \mapsto  \ell^\infty(\Theta,\mathbb{R}^{d-1})$,  be the least squares estimator with $d-1$ parameters, where $\mathcal{A}$ is a subset in which integration is well-defined. Let the $j$th coordinate of $\Upsilon(\theta,A)$ be given by:
\begin{align}
\Upsilon_j(\theta,A) = - 2 \int y_{t-j+1}(y_t - \theta_1 y_{t-1} - \ldots - \theta_{d-1} y_{t-d+1}) dA \equiv \int \upsilon_j(z,\theta) dA
\end{align}
Then, $\Upsilon(\cdot,A)$ is Hadamard differentiable at a distribution function $F$, tangentially to $C([-M,M]^d)$\footnote{Since $\upsilon_j(z,\theta)$ has bounded variation in the Hardy-Krause sense, this tangent set is sufficient for well-definition of the integral by an integration by parts argument.  For details, see Remark \ref{stochastic integral}.}, with derivative of the $j$th coordinate given by:
\begin{align}
\dot{\Upsilon}_{F,j}(\cdot, h) = \int \upsilon_j(z, \cdot ) dh
\end{align}
\end{proposition}

\begin{proof}
  Since the codomain is finite-dimensional, it suffices to establish Hadamard
  differentiability for each coordinate.  Since the estimating equations are
  symmetric, we can further restrict our attention to one coordinate. First for
  $||h_n - h||_\infty \rightarrow 0$ and $t_n \rightarrow 0$, point-wise in
  $\theta$, notice that:
\begin{align}
\frac{\int \upsilon_j(z,\theta) d(F +t_n h_n) - \int \upsilon_j(z,\theta) dF }{t_n} &= \frac{t_n \int \upsilon_j(z,\theta) dh_n }{t_n} +\frac{\int \upsilon_j(z,\theta) d(F-F)}{t_n}
\\ &\rightarrow \int \upsilon_j(z, \theta) dh
\end{align}
Next, since $\Theta$ is compact, $\int \upsilon_j(z,\cdot) dh_n$ is convex (affine), point-wise convergence implies uniform convergence \citep[Thm 3.1.4]{Hiriat-urrut-lemarechel-fundamentals-convex-analysis}.
Therefore,
\begin{align}
\norm{ \frac{\int \upsilon_{j}(z,\cdot) d(F +t_n h_n) - \int \upsilon_j(z,\cdot) dF}{t_n} - \dot{\Upsilon}_{F,j}(h,\cdot) }_\infty \rightarrow 0
\end{align}
and Hadamard differentiability of $\Upsilon(\cdot,A)$ at $F$ follows.
\end{proof}
See Appendix \ref{van der Vaart Wellner Modification} for additional details
regarding the multivariate stochastic integration that results from applying
the Functional Delta Method. In particular, Remark \ref{stochastic integral}
provides an additional sufficient condition on the loss function that ensures
that the pathwise Riemann-Stieltjes integral exists a.s.  This condition is
satisfied for squared error loss.
\end{proof}

\label{sec:ar_model_application}

\section{Simulation Study}
\label{Simulation Study}
\label{sec:simulations}
To demonstrate the performance of our method in a non-asymptotic setting, we
examine three simulated examples.  In each example, we will use squared error
loss, an $AR(\rho)$ model for prediction, and data generating process
satisfying the mixing conditions imposed in our theorem.  Note that here we
consider distributions with unbounded support; our procedure seems to work with
the distributions considered supported on entire real line.

\subsection{ARMA$(p,q)$ DGP}
Suppose that the data generating process is of the form:
\begin{align}
X_t = \sum_{i=1}^p \varphi_i X_{t-i} + \sum_{j=1}^q \theta_j \epsilon_{t-j} + \epsilon_t
\end{align}
where the coefficients satisfy stationarity and
$\{\epsilon_j\}_{j \in \mathbb{Z}}$ is an iid sequence such that
$E(\epsilon_j) = 0$ and the distribution of $\epsilon_j$ is dominated by the
Lebesgue measure.  Then \citet{Doukhan-on-mixing} establishes that $X_t$ is
geometrically $\beta$-mixing.

We will simulate from an ARMA(2,2) process, with coefficients
$\varphi_1 = 0.5$, $\varphi_2 = 0.5$, $\theta_1 = 0.5$, and $\theta_2 = 0.25$
with $\epsilon_i \sim N(0,1)$.  We will use an $AR(1)$ model for prediction.

\subsection{AR-ARCH DGP}

We will consider an AR data generating process with heteroscedastic ARCH noise.
That is, consider:
\begin{align*}
X_t &= \sum_{i=1}^p \varphi_i X_{t-i} + \epsilon_{t}
\\ \epsilon_t &= \sqrt{h_t} z_t
\\ h_t &= \omega + \sum_{j=1}^r \alpha_j \epsilon_{t-j}^2
\end{align*}
where $z_t$ is an iid noise sequence.  More specifically,
\citet{Lange-Rahbek-Jensen-on-AR-ARCH} show that an AR(1)-ARCH(1) model
satisfying $E(\log(\alpha_1 z_t^2)) < 0$ and $|\varphi_1| < 1$ is stationary
and geometrically $\beta$-mixing.

In our simulation, we will set $\varphi_1 = 0.8$, $\alpha_1 = 0.99$, and
$z_t \sim N(0,1)$.  Note that our choice of $\alpha_1$ leads to a high but
finite variance.  We will use an $AR(3)$ model for prediction.


\subsection{Markov-switching DGP}
Breaks in the data are commonly modeled as a switch from one stationary
generating process to another, where the regime switch is determined by a
Markov Chain.  Here, we consider a data generating process that switches
between a finite number of ARMA$(p,q)$ models.  In particular, we will
follow the example used by \citet{Lee-Markov-switching}:
\[
y_t =
\begin{cases}
 1.5 y_{t-1} + 0.6 e_{t-1} + e_t & \ X = 1\\
0.9 y_{t-1} - 1.2 e_{t-1} + e_t & \ X = 2 \\
0.7 e_{t-1} & \ X = 3 \\
\end{cases}
\]
with transition matrix  $P$ given by
\[ \left( \begin{array}{ccc}
0 & 0.2 & 0.8 \\
0.7 & 0 & 0.3 \\
0.5 & 0 & 0.5 \end{array} \right)
\]
If $e_t$ is iid and $\E |e_t| ^5 < \infty$, then $y_t$ is geometrically ergodic
and since it is a stationary Markov Chain, it is also geometrically
$\beta$-mixing.  We will simulate using $e_t \sim N(0,1)$ and an $AR(2)$ model
for prediction.

\subsection{Coverage of proposed bootstrap procedure}

We summarize the behavior of our estimator in Figure \ref{fig:coverage}.  For
all data generating processes, we examine 11 sample sizes, ranging from 50 to 1000, and select a block
length using \citet{PolitisWhite2004}.  We simulate a time series of length 1000
to estimate the risk.  We consider 500 bootstrap replications and
generate 500 different bootstrap confidence intervals based on different runs
of the data generating process to calculate coverage.

\begin{figure}[h]
\centering
 \includegraphics[scale=0.45]{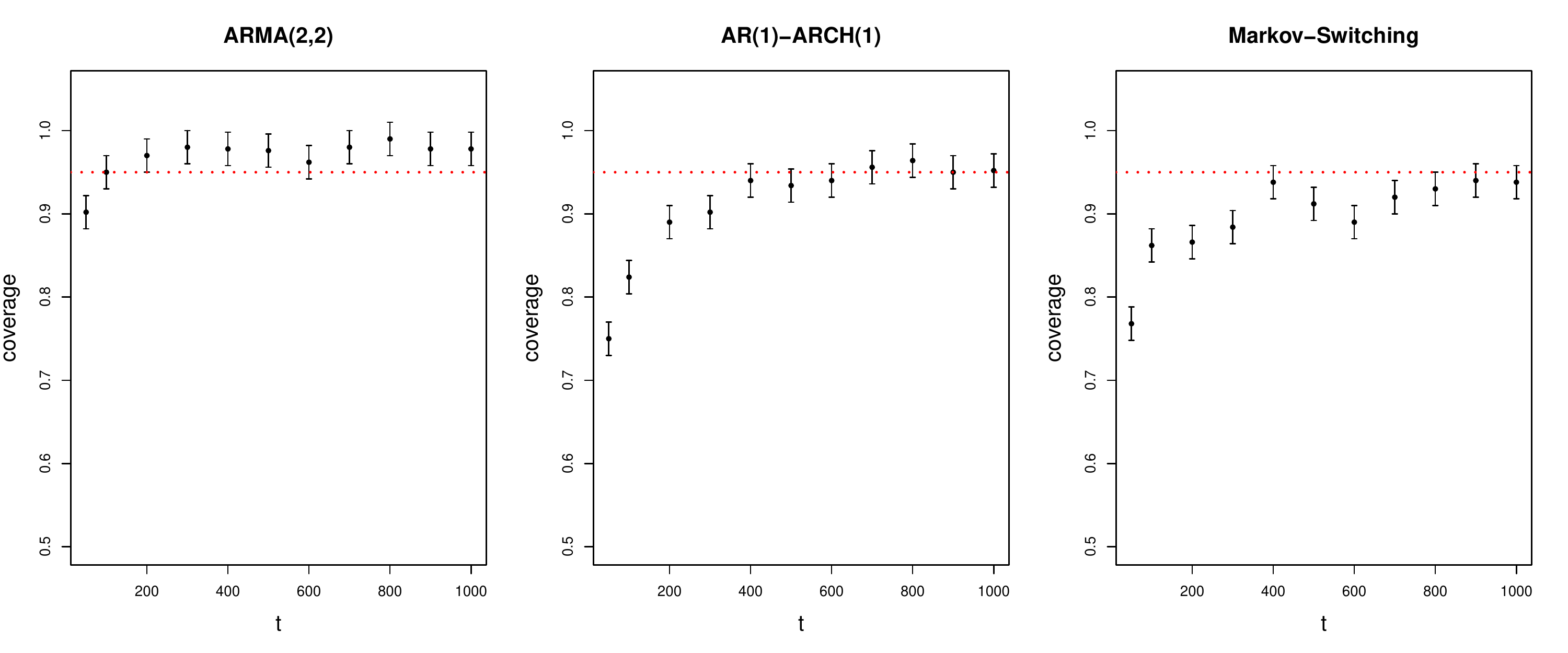}
\caption{Coverage of Bootstrapped Generalization Error Bounds}
\label{fig:coverage}
\end{figure}

For the AR(2,2) data generating process, we see that the bootstrap confidence
intervals have coverage close to desired $\alpha$-level even for sample sizes
as small as 25. Both the AR-ARCH(1) and Markov-switching data generating
processes exhibit poor coverage for small sample sizes but reach the desired
coverage level by 1000 observations.


\section{Discussion}
We propose a method based on the block bootstrap for constructing confidence
intervals for the risk. We succeed in showing bootstrap consistency with
minimal assumptions on the data generating process itself.  However, it may
seem that assumptions about mixing rates may be hard to verify.  While we do
not disagree with this sentiment entirely, the $\beta$-mixing coefficient is in
principle estimable \citep{risk-bounds-for-state-space-models}, and the cubic
mixing rate condition is satisfied for both $m$-dependent and geometrically
ergodic Markov Chains \citep{Bradley-strong-mixing}.


While we have not dealt directly with model selection, our work opens up
possibilities for that topic.  The difference of two risk functionals
satisfying the conditions of our theorem would also be Hadamard differentiable,
allowing for model selection, at least in the fixed-$d$ setting.  Our approach
could not distinguish between models with the same asymptotic risk, but this is
not a major concern if one presumes all models are mis-specified.  The
limitation to fixed $d$ may be more of an issue, and this is one more reason to
studying the growing-memory regime.

Another topic of future research is showing the viability of our procedure for
a wider range of models.  Showing Hadamard differentiability of more
complicated functionals may prove difficult, but
\citet{Hable-asymptotic-normality-of-kernel-methods} for instance is a
promising step in this direction.  In these cases, more conditions on the data
generating process may be needed, which may be a price worth paying to prove
guarantees for other models commonly used in data analysis.

\paragraph{Acknowledgments} We are grateful to Robert Kass, Daniel McDonald,
Alessandro Rinaldo and Valerie Ventura for helpful discussions.  We wish
to record special thanks to Prof.\ McDonald for his permission to develop
an idea proposed in \citep{McDonald-thesis}.

\appendix
\section{Proof Modifications}

At several points, we invoke propositions which are slight extensions or
modifications of ones established in the prior literature.  In the interest of
completeness, this appendix indicates how the published proofs need to be
adjusted.

\subsection{Proposition \ref{Bootstrap CLT}, after \citet[Theorem 1]{Radulovic-bootstrap-empirical-process-stationary}}
\label{Radulovic Modification}



Let $z_i = (y_i, \ldots y_{i+d-1})$.  By resampling these blocks of blocks, we
arrive at the CBB discussed in Section \ref{sec:method}.  After applying an
$f \in \cF$ to the $z_i$'s, we are again back to the case of real-valued
functions, and the rest of the argument carries through.



\subsection{Proposition \ref{CLT VC}, after \citet[Lemma 2.1]{Arcones-Yu-CLT}}
\label{Arcones Yu Modification}

Here, we need to confirm that the process
$\{\sqrt{t_0} (1/t_0 \sum_{i=t_0+1}^{N_0} f(X_i) - \Expect{f(X)}\}_\cF$ satisfies
the stochastic equicontinuity condition, defined in (\ref{stochastic
  equicontinuity}).  We will show this by making a small modification in Lemma
2.1, equation (2.14).  Permissibility implies that $\|\cdot \|_{\cF}$ is
measurable; therefore the following holds due to strict stationarity:
\begin{align}
&\ \ P( \|\sqrt{t_0} (1/t_0 \sum_{i=t_0+1}^{N} f(X_i) - \Expect{f(X)}\|_{\cF'(r;\|\cdot\|_p)} > \lambda ) \\
& = P( \|\sqrt{t_0} (1/t_0 \sum_{i=1}^{N-t_0} f(X_i) - \Expect{f(X)}\|_{\cF'(r;\|\cdot\|_p)} > \lambda )
\end{align}
At each $t$, the bound is the same as that of the process starting at $i=1$.

\subsection{Proposition \ref{Integration}, after \citet[Lemma 3.9.17]{van-der-Vaart-Wellner-weak-conv}}
\label{van der Vaart Wellner Modification}

The proof of \citet{van-der-Vaart-Wellner-weak-conv} generalizes to the
multivariate case if we consider a rectangular support region
$ \mathcal{R} \equiv \prod_{i=1}^d [a_i,b_i] \subset \mathbb{R}^d$.  In the
ensuing discussion of Lebesgue-Stieltjes and Riemann-Stieltjes integrals of the
form $\int f dg$, we will refer to $f$ as the integrand and $g$ as the
integrator. Lebesgue-Stieltjes integration requires that the function in the
integrand has bounded variation over the interval $[a,b]$.  In the univariate
case, this means that:
\begin{align}
\sup_{\mathcal{Y} \in \mathbb{Y}} \sum_{y \in \mathcal{Y}} | f(y^+) - f(y)| < \infty
\end{align}
where $\mathcal{Y}$ is a ladder on $[a,b]$, consisting of finitely many values
from this interval.  Suppose for a given ladder, we arrange each of the points
in increasing order: $y_0 < y_1, \ldots < y_m$.  The successor of a particular
$y$, denoted as $y^+$, is defined as the next element in the sequence.  The
supremum is over all such ladders.

A multivariate extension of Lebesgue-Stieltjes integrals requires a
generalization of this notion of bounded variation.  More than one such
generalization is possible; below we will discuss variation in the Hardy-Krause
sense, which will also require that we define variation in the Vitali sense.

We will start by introducing relevant concepts.  We will largely follow the
notation in \citet{Owen-multidimensional-variation}.  Let $[a,b]$ be a
hyperrectangle in $\mathbb{R}^d$. Suppose that $ u, v \subset \{1, \dots, d \}$
and $x,z \in [a,b]$ and $ u \cap v = \varnothing$.  Take $ x^u:z^v$ to be the
concatenation, resulting in a vector with values equal to $x$ for $i \in u$ and
equal to $z$ for $i \in v$.  Let $-v$ represent the complement of $v$.  Take
the d-fold alternating sum over $[a,b]$ to be
\begin{align}
\Delta(f;a,b) = \sum_{v \subset \{1, \ldots, d\} } (-1)^{|v|} f(a^v: b^{-v})
\end{align}
Now let $\mathcal{Y} = \prod_{j=1}^d \mathcal{Y}^j$, where $\mathcal{Y}^j$ is a
univariate ladder for the jth coordinate.  The variation of $f$ over
$\mathcal{Y}$ is:
\begin{align}
V_{\mathcal{Y}}(f) = \sum_{y \in \mathcal{Y}} |\Delta(f, y, y^+)|
\end{align}
By again taking the supremum of the variation over all possible ladders we
arrive at the Vitali notion of variation:
\begin{definition} The variation of $f$ on the hyperrectangle $[a,b]$ in the
  sense of Vitali, is:
  \begin{align}
    V_{\textbf{V}}(f) = \sup_{\mathcal{Y} \in \mathbb{Y}} V_{\mathcal{Y}}(f)
  \end{align}
\end{definition}

The Hardy-Krause notion of variation is closely related to that of Vitali
defined above.  It consists of restricting the function to take the value of
$b_i$ for each subset of the coordinates and summing the resulting Vitali
variations.
\begin{definition} The variation of $f$ on the hyperrectangle $[a,b]$ in the
  sense of Hardy and Krause is
\begin{align}
V_{\textbf{HK}}(f) =  \sum_{u \subset \{ 1,\ldots, d \}, u \neq \varnothing }  V_{\textbf{V}}(f(x^u : b^{-u}))
\end{align}
\end{definition}

Variation in the Hardy-Krause sense is generally larger than in the Vitali
sense; in fact, one can construct examples such that the former is infinite
while the latter is finite. See for example,
\citet{Beare-generalized-hoeffding}.

In the univariate case, bounded variation implies the existence of the Jordan
decomposition, where the integrator is expressed as the difference of two
monotonic functions.  Then, one can use the Caratheodory extension theorem to
uniquely match each Lebesgue-Stieltjes measure with an appropriate Lebesgue
measure.  An analogous result holds in the multivariate case if the integrator
is of bounded variation in the Hardy-Krause sense, making it a natural
condition to impose in our proof modification.

\begin{proposition}
Given a continuous function $A$ and a function of bounded variation $B$ in the sense of Hardy-Krause in the hyperrectangle $ \mathcal{R}\equiv \prod_{i=1}^d [a_i,b_i]$, define
\begin{align}
\phi(A,B) = \int A \ dB
\end{align}
Then, $\phi:  C(\mathcal{R}) \times BV_M (\mathcal{R}) \mapsto \R$ is Hadamard differentiable at each $(A,B) \in D_\phi$ such that $\int |dA| < \infty$.  The derivative is given by:
\begin{align}
\phi_{A,B}^\prime(\alpha, \beta) = \int A \ d \beta + \int \alpha \ dB
\end{align}
\end{proposition}
\begin{proof}
For $\alpha_t \rightarrow \alpha$ and $\beta_t \rightarrow \beta$, define $A_t = A + t \alpha_t$ and $B_t = B+ t \beta_t$.  Write:
\begin{align}
\frac{\int A_t dB_t - \int A dB}{t} - \phi_{A,B}^\prime (\alpha_t, \beta_t) = \int \alpha d(B_t - B) + \int (\alpha_t - \alpha) d(B_t - B)
\end{align}

Analogous to \citet{van-der-Vaart-Wellner-weak-conv}, the second term on the
RHS can be bounded by $ 2C \|\alpha_t - \alpha \|_\infty \rightarrow 0$ for some $C < \infty$.  See \citet{Aistleitner-Dick-bounded-variation-signed-measures} for the relationship between the total variation in the Hardy-Krause sense and variation of the signed measure. 

For the
first term, we again make an identical argument.  Since $\alpha$ is continuous we can construct a $d$-dimensional grid such that $\alpha$ varies no more
than $\epsilon$ within a particular grid.  Let $\widetilde{\alpha}$ be the
discretization that is constant and takes the value $\alpha(y)$ where $y$ is
the left endpoint of the grid. Then,
\begin{align*}
\left| \int \alpha d(B_t - B) \right| &\leq \|\alpha - \widetilde{\alpha} \|_\infty 2C \\
& \ + \sum_{y \in \mathcal{Y}_{\text{grid}}} \int_{[y,y^+]} |\widetilde{\alpha}(y)| d(B_t - B)
\end{align*}
where a coordinate is fixed at $b_i$ when there is no successor $y_i^+$.  The
first item on the RHS can be made arbitrarily small by making $\epsilon$ small,
and the last term is bounded by
$ |\mathcal{Y}_{\text{grid}}| \cdot \|\alpha\|_\infty \prod_{i=1}^d |b_i -a_i |
\cdot \|B_t - B \|_\infty$, which can be made arbitrarily small for a fixed
partition.
\end{proof}

\begin{remark}
\label{stochastic integral}
Evaluating the Hadamard derivative after applying the Functional Delta Method
results in integration with respect to a Gaussian process.  This leads to some
issues as even the canonical empirical process
$\mathbb{G} = \lim_{t \rightarrow \infty} \sqrt{t} (F_t -F)$ is known to have
unbounded variation in the iid one-dimensional case for general $P$ even though
it is uniformly continuous; see \citet{Dudley-p-variation}.

We can follow \citet{van-der-Vaart-Wellner-weak-conv} and define the integral
via an integration by parts formula. Since tight Gaussian processes are
continuous a.s., it is enough to require that the integrand is of bounded
variation in the Hardy-Krause sense for the pathwise Riemann-Stieltjes integral
to exist a.s.  A simple condition to ensure bounded variation is that the
integrand is differentiable with bounded mixed partial derivatives
\citep[Prop.\ 13]{Owen-multidimensional-variation}.

Lebesgue-Stieltjes integration is often thought of as a generalization of
Riemann-Stieltjes integration.  This is indeed the case when the integrator has
bounded variation.  Riemann-Stieltjes integration requires that the integrand
and integrator share no points of discontinuity, whereas discontinuities are
not an issue when the Lebesgue-Stieltjes integral is transformed into a
Lebesgue integral.

However, an integration by parts formula for Lebesgue-Stieltjes integrals
generally requires both functions to be of bounded variation since the strategy
is to replace the integrand with a corresponding measure and use the
Fubini-Tonelli Theorem. A Riemann-Stieltjes integral can be shown to exist when
one side of the integration by parts formula exists, which does not require
both functions to be of bounded variation \citep{Prause-Steland-sequential-detection-3d}.
\end{remark}

Note that in the main theorem, we will fix $\theta$ for $L(z, \theta)$ in this step in the Chain Rule.

\section{Implications for K-fold Cross-Validation}
\label{sec:k-fold-cv}

\subsection{Functional Representation of K-fold Cross-Validation}
\label{functional representation}

Here we show that K-fold cross-validation can also be expressed as a functional
of distribution functions. Let $F_{t,i}$ correspond to the empirical
distribution function of the ith fold at time $t$.  We will ignore the issue of
fold membership when $t$ is not divisible by $k$ since this is irrelevant in
the triangular array setting.  A simple extension of Lemma \ref{joint
  convergence lemma} implies that:
\begin{align}
 \sqrt{t}
\begin{pmatrix}
 F_{t,1} - F \ \\
 \vdots \\
 F_{t,k} - F
 \end{pmatrix}
 \rightsquigarrow \mathbb{H}_k
\end{align}
where $\mathbb{H}_k$ is the $k$th product of $\mathbb{G}$, the limiting Gaussian process with covariance function given in (\ref{covariance function}).  Since we have weak convergence, we can apply the Functional Delta Method to show asymptotic Normality and consistency of Cross Validation under nearly identical conditions to the main theorem:

\begin{proposition}
Consider the Cross-Validated risk, defined as:
\begin{align}
R_{CV}(\hat{f}) = \frac{1}{k} \sum_{i=1}^k  \int L \left(z, \hat{\theta}\left(\frac{1}{k-1}\sum_{j \neq i} F_i \right) \right) dF_i
\end{align}
Suppose that conditions \ref{mixing condition}, \ref{loss condition}, \ref{estimator condition} are satisfied and that $K$ is a fixed constant. Then $\sqrt{t} \left[R_{CV}(\hat{f}) - R(\hat{f}) \right] \rightsquigarrow N(0, \sigma^2)$ for some $\sigma^2 < \infty$.
\end{proposition}

We study the asymptotic Normality of the cross-validated risk using the same
data generating processes considered in Section \ref{Simulation Study} in the
plots below.  For each process, we run 5-fold cross-validation using an AR(2)
model for prediction on a range of sample sizes.  We examine asymptotic
Normality by simulating 10000 runs of the procedure and comparing the
standardized quantiles with those of a Normal distribution.  In general, we see
that the Normal approximation is poor for relatively small sample sizes
$(n = 50, n=100)$, with the quantiles exhibiting very heavy tails.  However,
with increasing $n$, we see that the tails become better behaved.  

\begin{figure}[H]
	\begin{center}$
	\begin{array}{lll}
        \includegraphics[width=.3\linewidth]{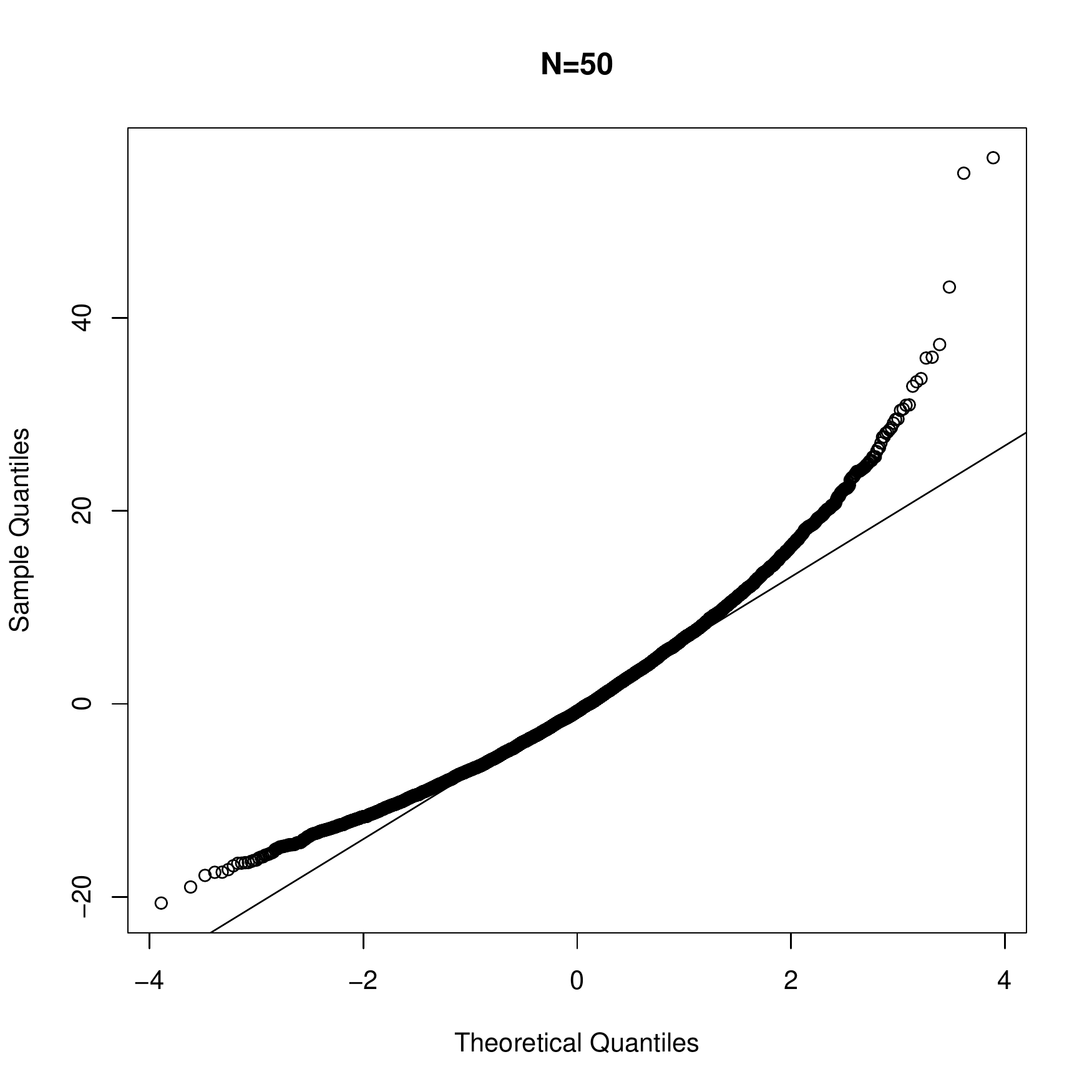}
        \includegraphics[width=.3\linewidth]{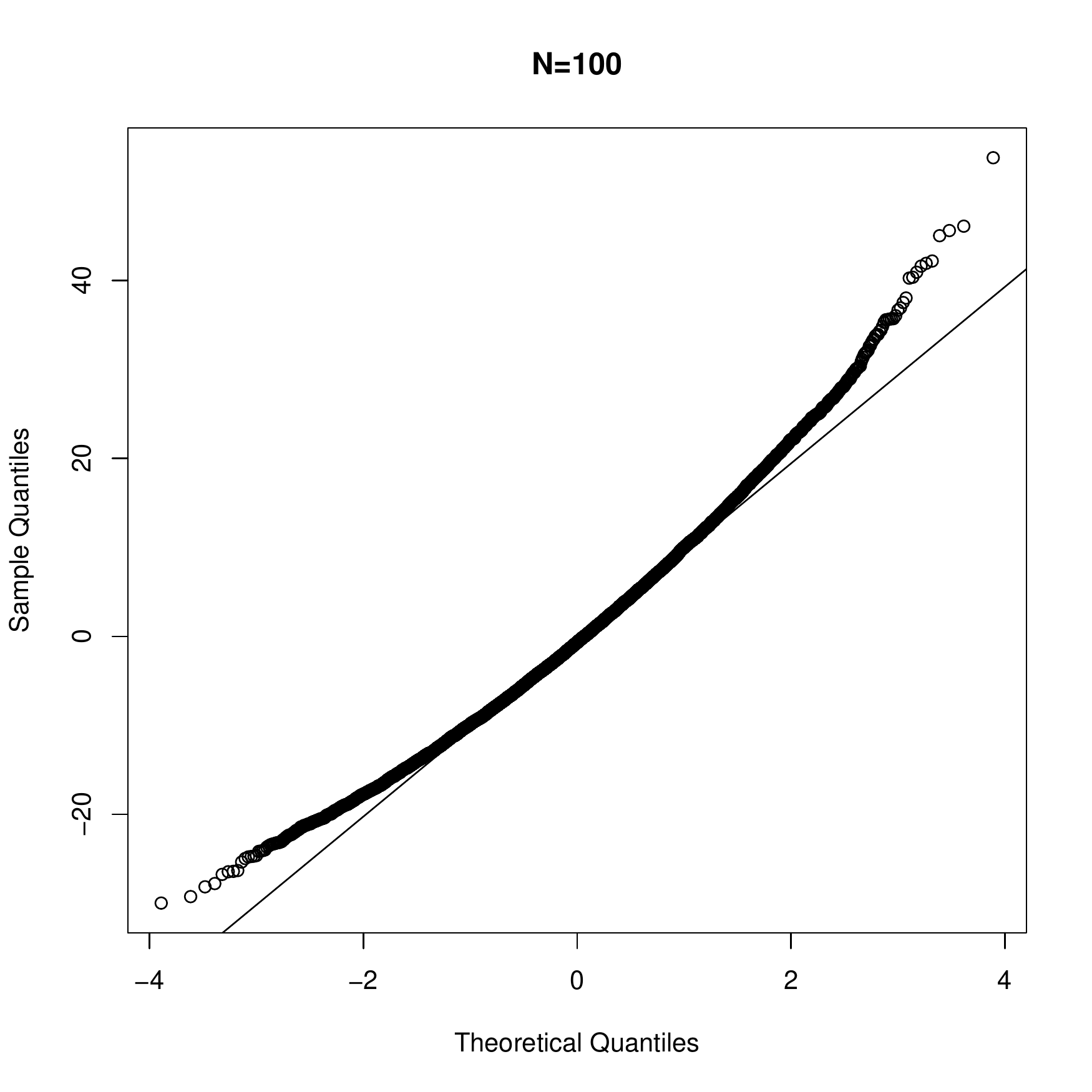}
        \includegraphics[width=.3\linewidth]{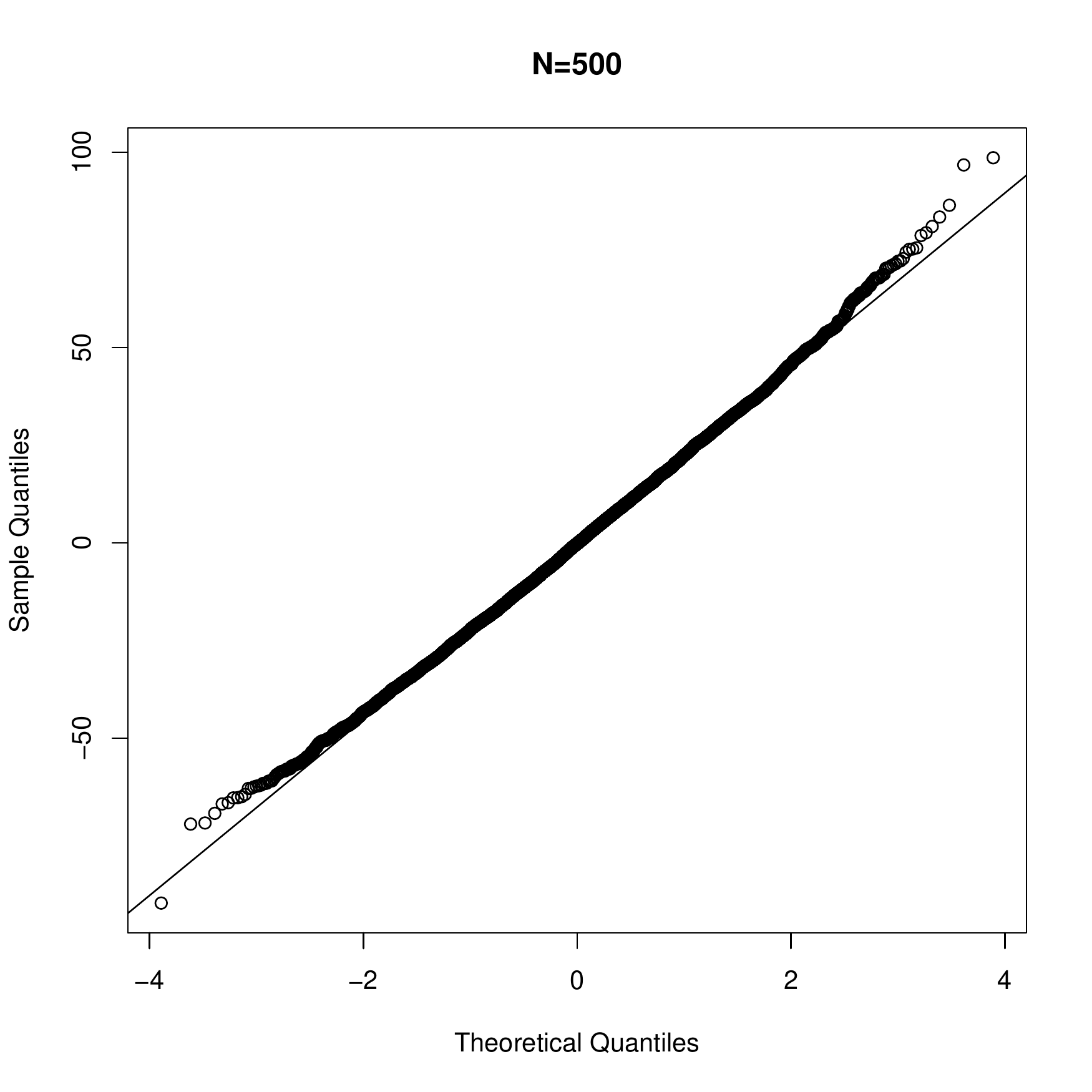}
    \end{array}$
    \end{center}

    \begin{center}$
    \begin{array}{lll}
        \includegraphics[width=.3\linewidth]{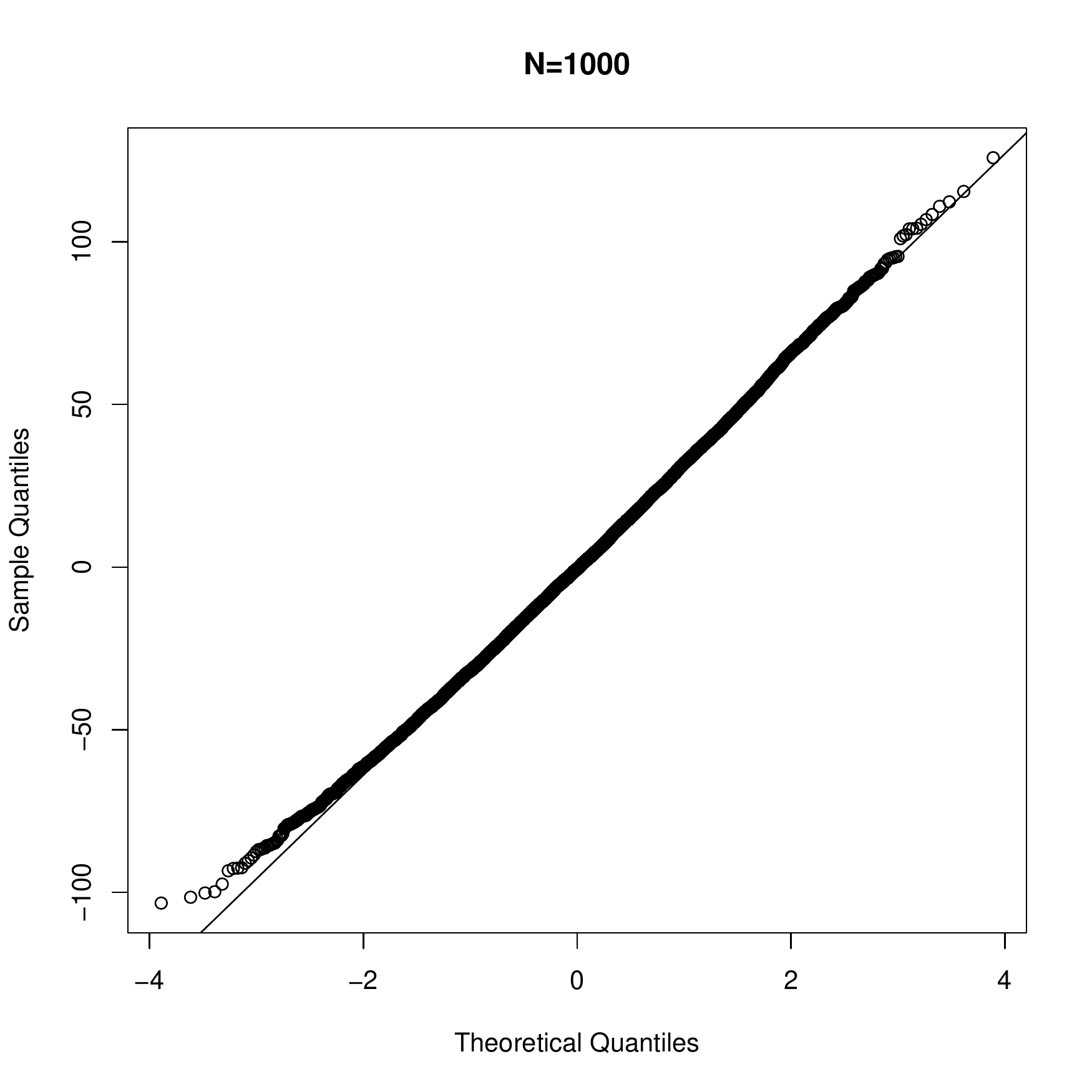}
        \includegraphics[width=.3\linewidth]{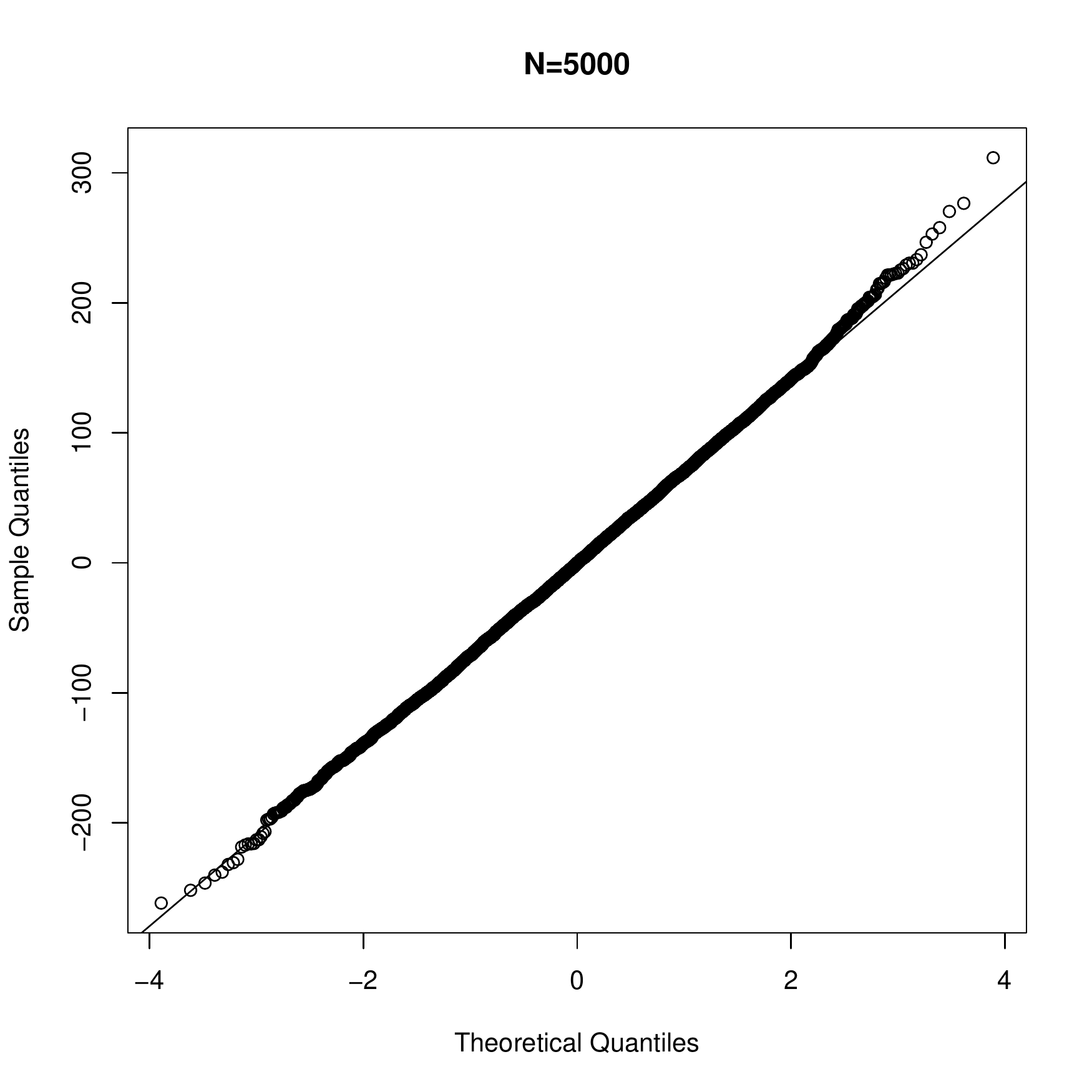}
        \includegraphics[width=.3\linewidth]{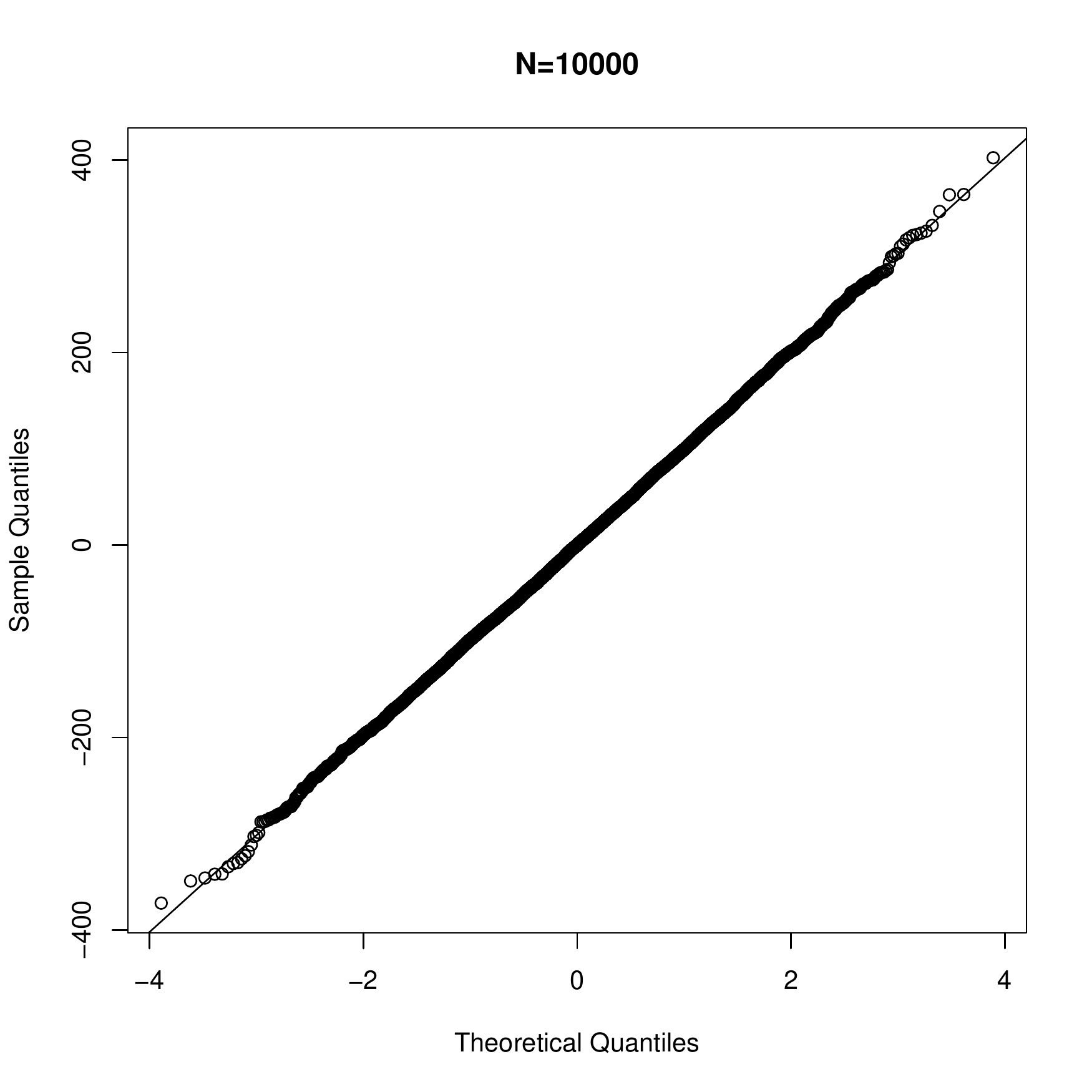}
    \end{array}$
    \end{center}
    \caption{ARMA(2,2) process}
    \end{figure}

\begin{figure}[H]
	\begin{center}$
	\begin{array}{lll}
        \includegraphics[width=.3\linewidth]{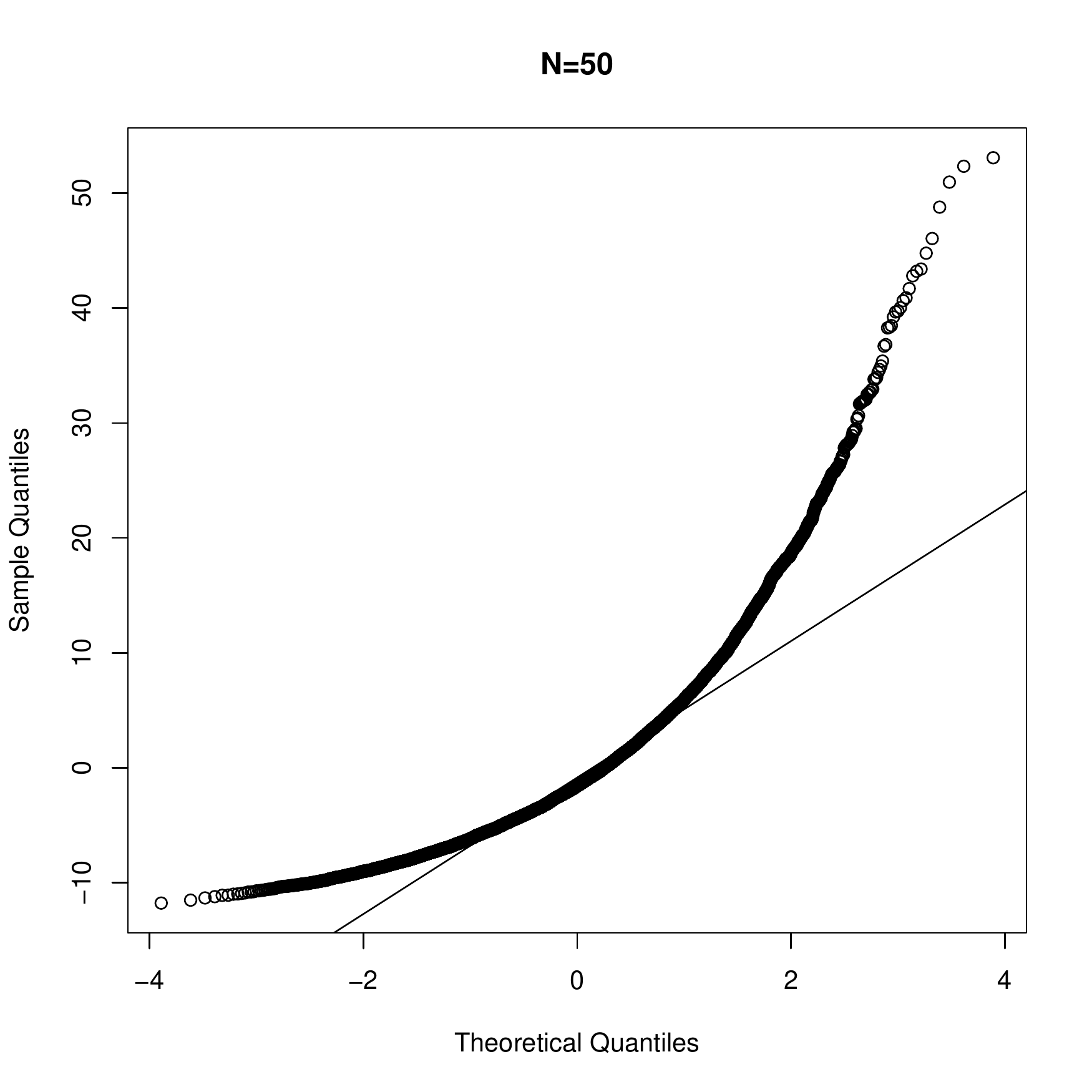}
        \includegraphics[width=.3\linewidth]{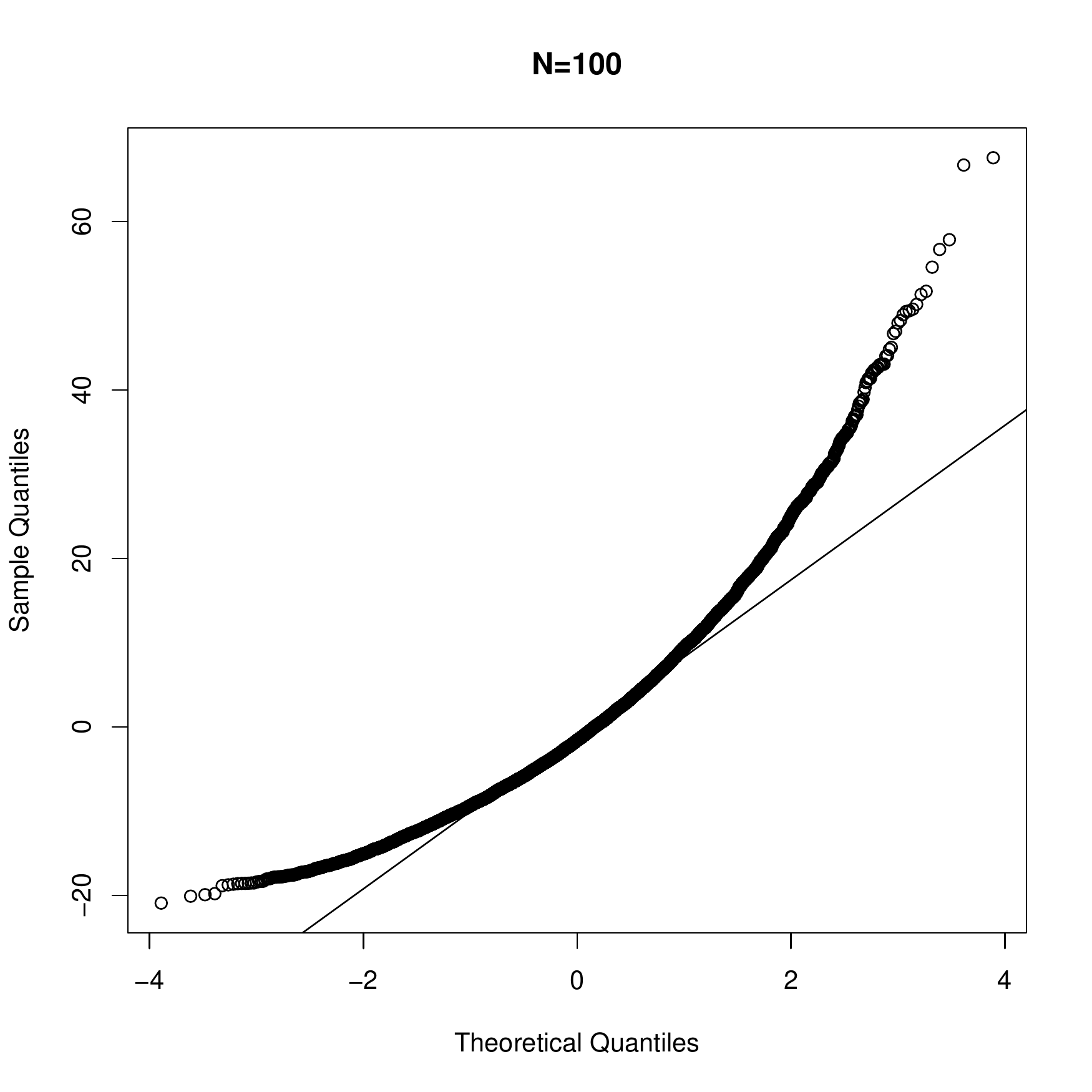}
        \includegraphics[width=.3\linewidth]{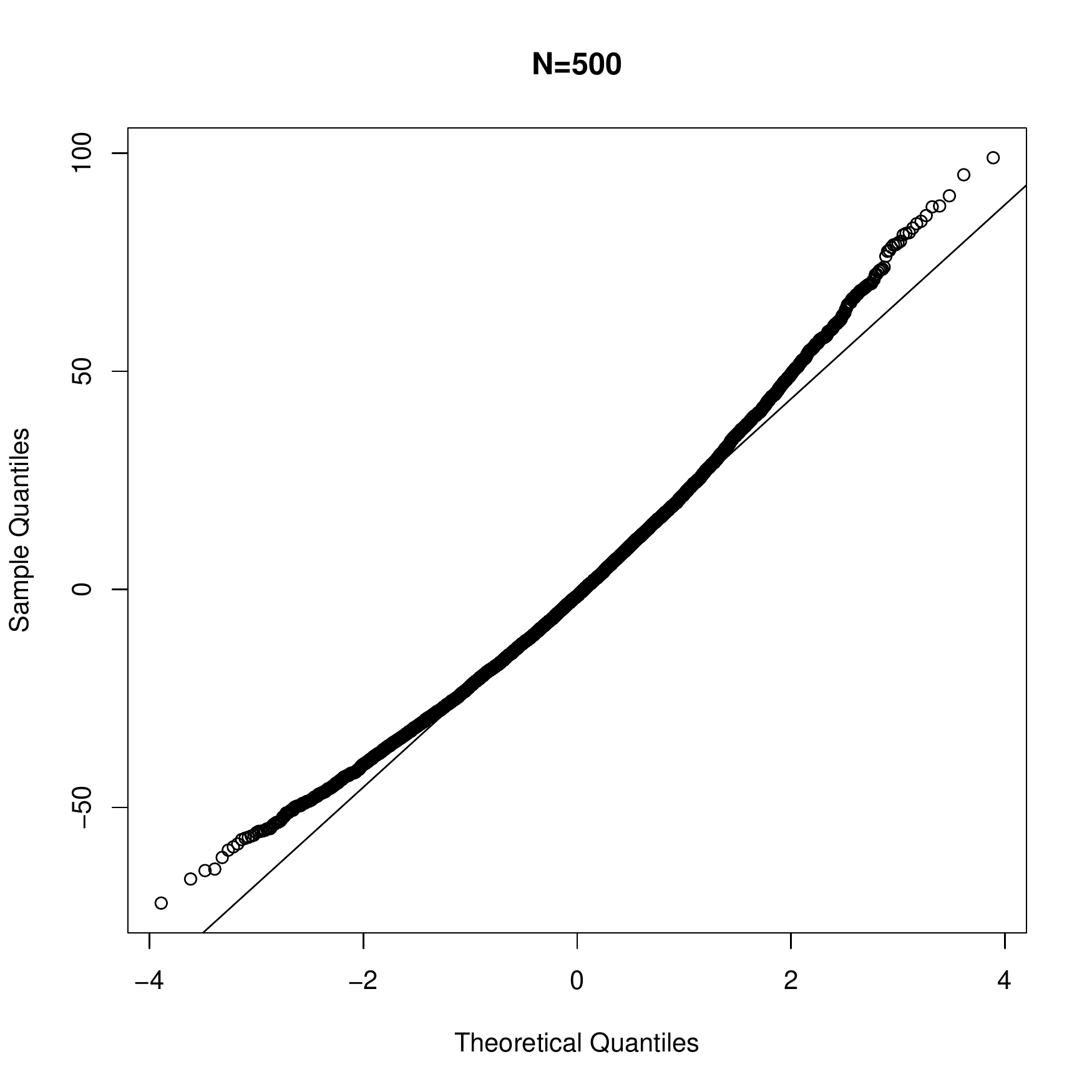}
    \end{array}$
    \end{center}

    \begin{center}$
    \begin{array}{lll}
        \includegraphics[width=.3\linewidth]{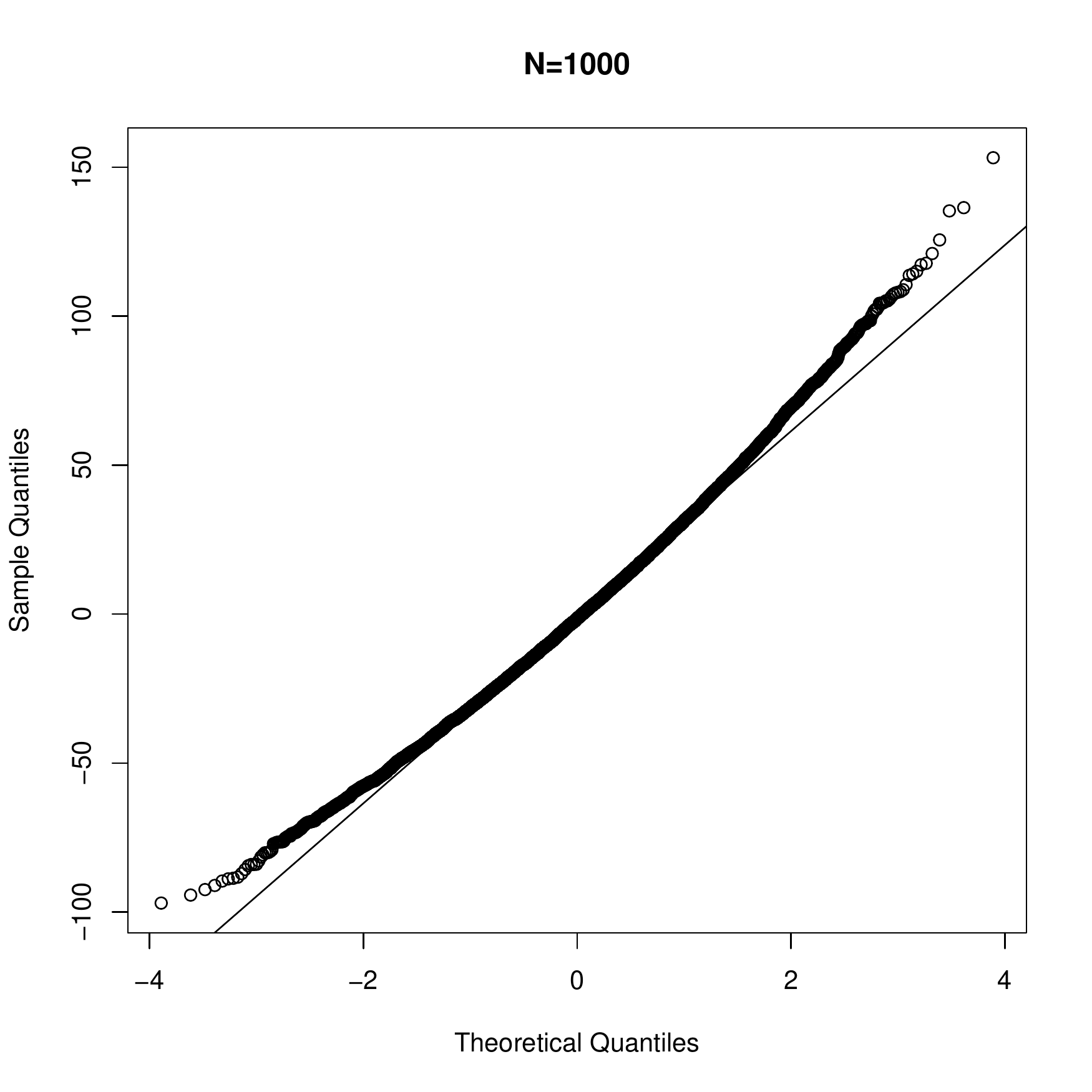}
        \includegraphics[width=.3\linewidth]{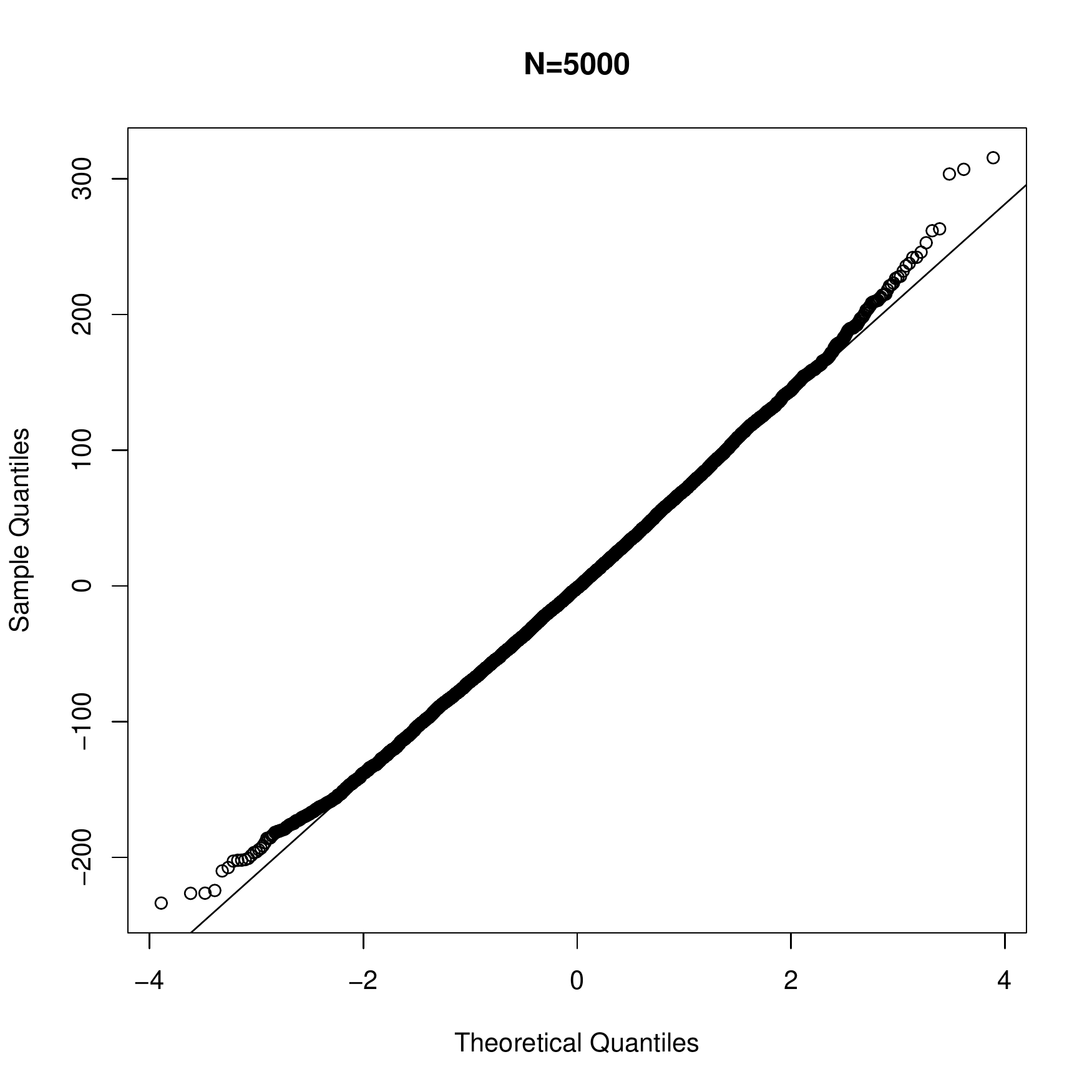}
        \includegraphics[width=.3\linewidth]{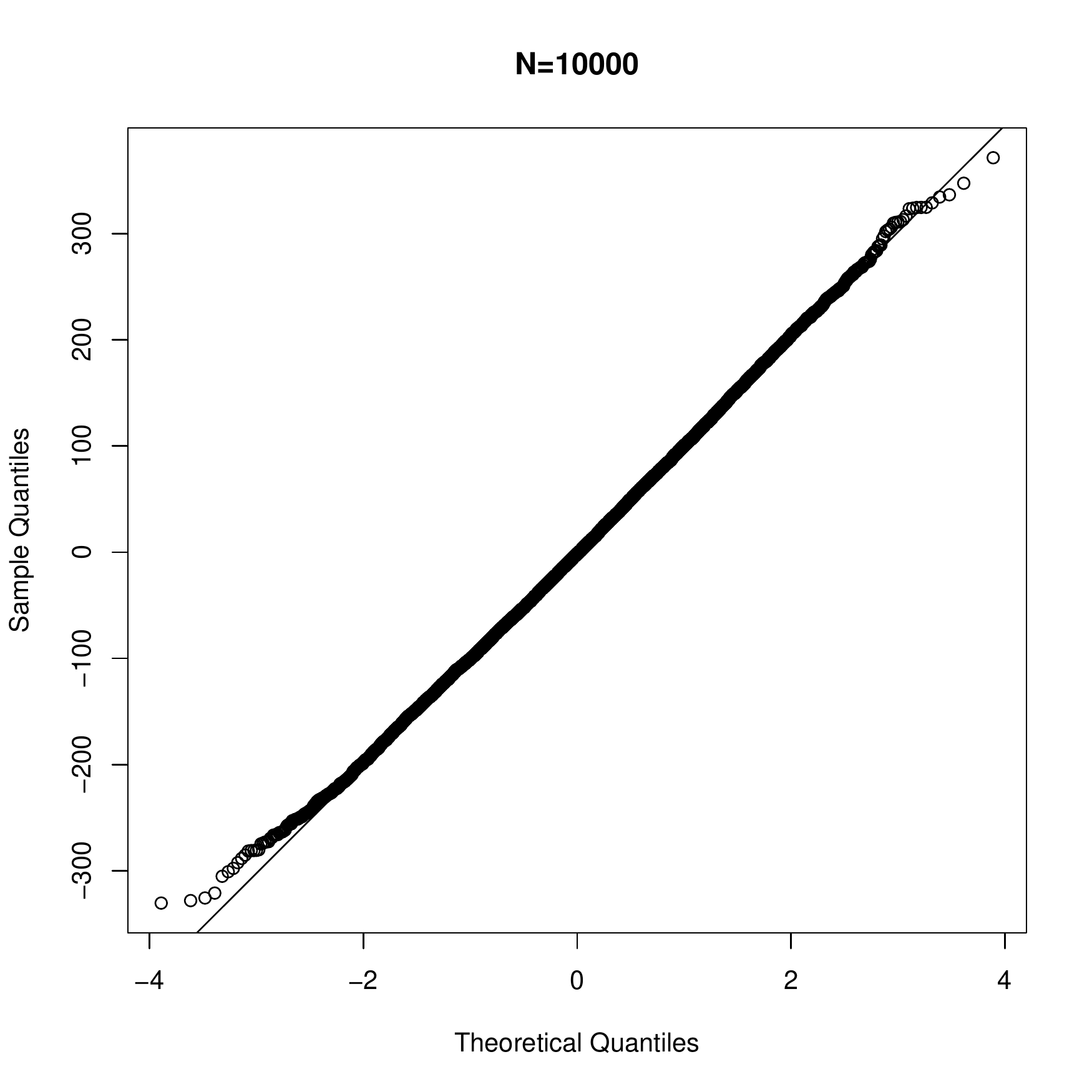}
    \end{array}$
    \end{center}
    \caption{AR(1)-ARCH(1) process}
  \end{figure}

To our knowledge, the only other result regarding the risk consistency of Cross-Validation for time series is \citet{Racine-consistent-cv-for-dependent-data}, who considers autoregressive models.  Our result is more agnostic about the data generating process, but could probably be extended using different tools.

 \bibliographystyle{imsart-nameyear}
\bibliography{locusts}

\end{document}